\documentclass[11pt]{amsart}
\usepackage{latexsym, amssymb, stmaryrd, eucal}
\usepackage[T1]{fontenc}

\DeclareTextSymbol{\thh}{T1}{254}
\def\th{\textnormal{\thh}}
\newtheorem{thm}{Theorem}[subsection]
\newtheorem{lemma}[thm]{Lemma}
\newtheorem{prop}[thm]{Proposition}
\newtheorem{cor}[thm]{Corollary}

\newtheorem{result}[thm]{Result}

\theoremstyle{definition}
\newtheorem{df}[thm]{Definition}
\newtheorem{rmk}[thm]{Remark}
\newtheorem{rmks}[thm]{Remarks}
\newtheorem{ex}[thm]{Example}
\newtheorem{fact}[thm]{Fact}
\newtheorem{question}[thm]{Question}

\newcommand{\BB}{\mathbb{B}}  
\newcommand{\BK}{\mathbb{K}}  
\newcommand{\BN}{\mathbb{N}}

\newcommand{\cu}[1]{\mathcal{#1}}

\newcommand{\bo}[1]{\boldsymbol{#1}}

\newcommand{\ti}[1]{\widetilde{#1}}

\newcommand{\sa}[1]{\mathsf{#1}}

\renewcommand{\hat}{\widehat}

\makeatletter

\def\indsym#1#2{%
  \setbox0=\hbox{$\m@th#1x$}%
  \kern\wd0%
  \hbox to 0pt{\hss$\m@th#1\mid$\hbox to 0pt{$\m@th#1^{#2}$}\hss}%
  \lower.9\ht0\hbox to 0pt{\hss$\m@th#1\smile$\hss}%
  \kern\wd0}
\newcommand{\ind}[1][]{\mathop{\mathpalette\indsym{#1}}}

\def\nindsym#1#2{%
  \setbox0=\hbox{$\m@th#1x$}%
  \kern\wd0%
  \hbox to 0pt{\hss$\m@th#1\not$\kern1.4\wd0\hss}
  \hbox to 0pt{\hss$\m@th#1\mid$\hbox to 0pt{$\m@th#1^{\,#2}$}\hss}%
  \lower.9\ht0\hbox to 0pt{\hss$\m@th#1\smile$\hss}%
  \kern\wd0}
\newcommand{\nind}[1][]{\mathop{\mathpalette\nindsym{#1}}}

\def\dotminussym#1#2{%
  \setbox0=\hbox{$\m@th#1-$}%
  \kern.5\wd0%
  \hbox to 0pt{\hss\hbox{$\m@th#1-$}\hss}%
  \raise.6\ht0\hbox to 0pt{\hss$\m@th#1.$\hss}%
  \kern.5\wd0}
\newcommand{\dotminus}{\mathbin{\mathpalette\dotminussym{}}}

\def \<{\langle}
\def \>{\rangle}
\def \((  {(\!(}
\def \)) {)\!)}

\def \cl {\operatorname{cl}}

\def \tp{\operatorname{tp}}

\def \acl{\operatorname{acl}}
\def \dcl{\operatorname{dcl}}

\def \DLO{\operatorname{DLO}}

\def \APr{\operatorname{APr}}
\def \div{\operatorname{div}}
\def \thdiv{\th\operatorname{-div}}
\def \thfork{\th\operatorname{-fork}}

\def \as {\operatorname{a.s}}
\def \fo{\operatorname{fdcl}}

\numberwithin{equation}{section}
\def \l{\llbracket}
\def \rr{\rrbracket}

\def\thind{\ind[\th]}
\def\nthind{\nind[\th]}

\allowdisplaybreaks[2]

\begin{document}

\title{Independence Relations in Randomizations}

\author{Uri Andrews, Isaac Goldbring, and H. Jerome Keisler}

\address{University of Wisconsin-Madison, Department of Mathematics, Madison,  WI 53706-1388}
\email{andrews@math.wisc.edu}
\urladdr{www.math.wisc.edu/~andrews}
\email{keisler@math.wisc.edu}
\urladdr{www.math.wisc.edu/~keisler}
\address {University of Illinois at Chicago, Department of Mathematics, Statistics, and Computer Science, Science and Engineering
Offices (M/C 249), 851 S. Morgan St., Chicago, IL 60607-7045, USA}
\email{isaac@math.uic.edu}
\urladdr{www.math.uic.edu/~isaac}

\date{September 3, 2014}

\begin{abstract}
The randomization of a complete first order theory $T$ is the complete continuous theory $T^R$
with two sorts, a sort for random elements of models of $T$, and a sort for events in an underlying probability space.  We study various notions of independence in models of $T^R$.
\end{abstract}

\maketitle

\section{Introduction}

A randomization of a first order structure $\cu M$, as introduced by Keisler [Ke1] and formalized as a metric structure by Ben Yaacov and Keisler [BK], is a continuous structure $\cu N$ with two sorts, a sort for random elements of $\cu M$, and a sort for events in an underlying atomless probability space. Given a complete first order theory $T$, the theory $T^R$ of randomizations of models of $T$ forms a complete theory in continuous logic, which is called the randomization of $T$.  In a model $\cu N$ of $T^R$, for each $n$-tuple $\vec{ a}$ of random elements and each first order
formula $\varphi(\vec v)$, the set of points in the underlying probability space where $\varphi(\vec{ a})$ is true is an event
denoted by $\l\varphi(\vec{ a})\rr$.

In general, the theories $T$ and $T^R$ share many model-theoretic features.  In particular, it was shown in [BK] that $T$ is stable if and only if $T^R$ is stable.  Unfortunately, the analogous result for simplicity in place of stability is false, as was shown in [Be2].  More precisely, either $T^R$ is dependent (in which case, if it is simple, then it is stable) or else it is not simple.

Recall that a classical (resp. continuous) theory is \emph{rosy} if $T^{\operatorname{eq}}$ possesses a strict (countable) independence relation; in this case, there is a weakest such notion of independence, namely \emph{thorn independence}, a notion first introduced by Thomas Scanlon.  (See Section 2 below for precise definitions.)  Classical rosy theories were first studied in the theses of Alf Onshuus and Clifton Ealy as a common generalization of simple theories and o-minimal theories.  Continuous rosy theories were first studied in the paper of Ealy and the second author in [EG].  In light of the previous paragraph, it is natural to ask:  is $T$ rosy if and only if $T^R$ is rosy?

Motivated by the above question, in this paper, we begin studying various notions of independence in models of $T^R$, including algebraic independence, dividing independence, and thorn independence.  We also study the ``pointwise'' version of a notion of independence, that is, the notion of independence on models of $T^R$ obtained by asking for almost everywhere independence.


We conclude this introduction with an outline of the rest of the paper.  In Section 2, we recall the relevant background from continuous logic as well as the general theory of abstract independence relations as exposited in [Ad2].  In Section 3, we introduce the notion of countably based independence relations.  For every ternary relation $\ind[I]$ with monotonicity, there is a unique countably based relation that agrees with $\ind[I]$ on countable sets.   This will aid us in defining, for a given ternary relation on small subsets of the big model of $T$, a corresponding pointwise notion.  In Section 4, we recall some basic facts about randomizations as well as the results from [AGK] concerning definable and algebraic closure in models of $T^R$.

In Section 5, we begin the study of notions of independence in models of $T^R$ in earnest.  We first prove some downward results, culminating in Corollary 5.1.3, which states that if definable and algebraic closure coincide in models of $T$ and $T^R$ is real rosy, then $T$ is real rosy.  (Real rosiness requires that $T$ has a strict independence relation, while rosiness requires that $T^{\operatorname{eq}}$ has such a relation.)  We then move on to studying notions of independence, first on the event sort, and then on the random variable sort.  Section 5 concludes with a study of the properties of algebraic independence in $T^R$.


Section 6 is concerned with notions of pointwise independence.  Given a ternary relation $\ind[I]$ with monotonicity on models of $T$, $\ind[I\omega] \ \ $ is the countably based relation on small subsets of the big model of $T^R$ such that for all countable  $A,B,C$, $A\ind[I\omega]_C \ B$ holds if and only if $A(\omega)\ind[I]_{C(\omega)} B(\omega)$ holds for almost all $\omega$ in the underlying probability space.  The results of Section 3 guarantee the unique existence of $\ind[I\omega] \ \ $.  We then make a detailed study of the pointwise notions of independence stemming from algebraic independence, dividing independence, and thorn independence.


Continuous model theory in its current form is developed in the papers [BBHU] and [BU].  Randomizations of models are treated in [AGK], [AK], [Be], [BK], [EG], [GL], and [Ke1].

\section{Preliminaries on Continuous Logic}

We will follow the notation and terminology of [BK] and [AGK].
We assume familiarity with the basic notions about continuous model theory as developed
in [BBHU], including the notions of a theory, structure, pre-structure,  model of a theory,
elementary extension, isomorphism, and $\kappa$-saturated structure.
In particular, the universe of a pre-structure is a pseudo-metric space, the universe of a
structure is a complete metric space, and every pre-structure has a unique completion.
A \emph{tuple} is a finite sequence, and $A^{<\BN}$ is the set of all tuples of elements of $A$.
In a metric space or continuous structure, the closure of a set $C$ is denoted by $\cl(C)$.
We use the word ``countable'' to mean of cardinality at most $\aleph_0$.
We assume throughout that $L$ is a countable first order signature,
and that $T$ is a complete theory for $L$ whose models have at least two elements.
We will sometimes write $\varphi(A)$ for a first order formula with finitely many parameters in a set $A$,
and use similar notation for more than one parameter set.

\subsection{Types and Definability}

For a first order structure $\cu M$ and a set $A$ of elements of $\cu M$, $\cu M_A$ denotes the structure formed by
adding a new constant symbol to $\cu M$ for each $a\in A$.
The \emph{type realized by} a tuple $\vec b$ over the parameter set $A$ in $\cu M$ is the set $\tp^\cu M(\vec b/A)$ of formulas
$\varphi(\vec u,{\vec a})$
with $\vec a\in A^{<\BN}$ satisfied by $\vec b$ in $\cu M_A$.  We call $\tp^{\cu M}(\vec b/A)$ an \emph{$n$-type} if $n=|\vec b|$.

In the following, let $\cu N$ be a metric structure and let $ A$ be a set of elements of $\cu N$.  $\cu N_{ A}$ denotes the structure formed by
adding a new constant symbol to $\cu N$ for each $ a\in  A$.
The \emph{type} $\tp^\cu N(\vec{ b}/ A)$ \emph{realized} by $\vec { b}$ over the parameter set $ A$ in $\cu N$
is the function $p$ from formulas to $[0,1]$ such that
for each formula $\Phi(\vec{x},\vec { a})$ with $\vec{ a}\in { A}^{<\BN}$, we have
$\Phi(\vec{x},\vec { a})^p=\Phi(\vec { b},\vec { a})^\cu N$.

We now recall the notions of definable element and algebraic element from [BBHU].
An element ${ b}$ is \emph{definable over} $ A$ in $\cu N$, in symbols $ b\in\dcl^\cu N( A)$,
if there is a sequence of formulas $\langle\Phi_k(x,\vec{ a}_k)\rangle$ with $\vec{ a}_k\in{ A}^{<\BN}$ such
that the sequence of functions $\langle\Phi_k(x,\vec{ a}_k)^\cu N\rangle$ converges uniformly in $x$ to the distance function $d(x, b)^\cu N$
of the corresponding sort.

When $b$ is an element and $C$ is a set in $\cu N$, the distance $d(b,C)$ is defined by $d(b,C)=\inf_{c\in C} d(b,c)$, with the
convention that $d(b,\emptyset)=1$.
$ b$ is \emph{algebraic over $ A$} in $\cu N$, in symbols $ b\in\acl^\cu N( A)$, if there is a compact set $C$ and
a sequence of formulas $\langle\Phi_k(x,\vec { a}_k)\rangle$ with $\vec{ a}_k\in{ A}^{<\BN}$ such
that $b\in C$ and the sequence of functions $\langle\Phi_k(x,\vec{ a}_k)^\cu N\rangle$  converges uniformly in $x$ to the distance function
$d(x,C)$ of the corresponding sort.

If the structure $\cu N$ is clear from the context, we will sometimes drop the superscript and write $\tp, \dcl, \acl$ instead of $\tp^\cu N, \dcl^\cu N, \acl^\cu N$.  We will often use the following facts without explicit mention.

\begin{fact}  \label{f-definable} ([BBHU], Exercises 10.7 and 10.10)  For each element $ b$ of $\cu N$,
the following are equivalent, where $p=\tp^\cu N( b/ A)$:
\begin{enumerate}
\item $ b$ is definable over $ A$ in $\cu N$;
\item in each model $\cu N'\succ\cu N$, $ b$ is the unique element that realizes $p$ over $ A$;
\item $ b$ is definable over some countable subset of $ A$ in $\cu N$.
\end{enumerate}
\end{fact}

\begin{fact}  \label{f-algebraic} ([BBHU], Exercise 10.8 and 10.11)
For each element $ b$ of $\cu N$,
the following are equivalent, where $p=\tp^\cu N( b/ A)$:
\begin{enumerate}
\item $ b$ is algebraic over $ A$ in $\cu N$;
\item in each model $\cu N'\succ\cu N$, the set of elements $ b$ that realize $p$ over $ A$ in $\cu N'$ is compact.
\item $ b$ is algebraic over some countable subset of $ A$ in $\cu N$.
\end{enumerate}
\end{fact}

\begin{fact}  \label{f-algebraic-cardinality}  (Follows from [BBHU], Exercise 10.8)
For every set $A$, $\acl(A)$ has cardinality at most $( |A|+2)^{\aleph_0}$.
\end{fact}

\begin{fact} \label{f-definableclosure}  (Definable Closure, Exercises 10.10 and 10.11, and Corollary 10.5 in [BBHU])
\begin{enumerate}
\item If $ A\subseteq\cu N$ then $\dcl( A)=\dcl(\dcl( A))$ and $\acl( A)=\acl(\acl( A))$.
\item If $ A$ is a dense subset of $\cl(B)$ and $ B\subseteq\cu N$, then $\dcl(A)=\dcl( B)$ and $\acl(A)=\acl( B)$.
\item If $A$ is separable, then $\dcl(A)$ and $\acl(A)$ are separable.
\end{enumerate}
\end{fact}

It follows that for any $ A\subseteq\cu N$, $\dcl( A)$ and $\acl( A)$ are closed with respect to the metric in $\cu N$.

\subsection{Abstract Independence Relations}

Since the various properties of independence are given some slightly different names in
various parts of the literature, we take this opportunity to declare that we are following the terminology established in [Ad2], which is
repeated here for the reader's convenience.  In this paper, we will sometimes write $AB$ for $A\cup B$, and write $[A,B]$ for
$\{D\colon A\subseteq D \wedge D\subseteq B\}$  (we do not use $(A,B)$ in the analogous way).
We assume throughout this paper that $\upsilon$ is an uncountable inaccessible cardinal that is held fixed.  By a
\emph{big model} of a complete theory $T$ with infinite models we mean a saturated model $\cu N\models T$ of cardinality $|\cu N|=\upsilon$.
For a complete theory $T$ with finite models, we call every model of $T$ big.
Thus every complete theory has a unique big model up to isomorphism.  For this reason, we sometimes refer to ``the'' big model of a complete theory $T$.
We call a set \emph{small} if it has cardinality $< \upsilon$, and \emph{large} otherwise.
Note that by Fact \ref{f-algebraic-cardinality}, the algebraic closure of every small set is small.

If you wish to avoid the assumption that uncountable inaccessible cardinals exist, you can instead assume only that $\upsilon=\upsilon^{\aleph_0}$
and take a big model to be an $\upsilon$-universal domain, as in [BBHU], Definition 7.13.  With that approach, a big model exists but is not unique.

\begin{df}[Adler]  Let $\cu N$ be the big model of a continuous or first order theory.  By a \emph{ternary relation over $\cu N$} we mean
a ternary relation $\ind$ on the small subsets of $\cu N$.  We say that $\ind$
is an \emph{independence relation} if it satisfies the following \emph{axioms for independence relations} for all small sets:
\begin{enumerate}
\item (Invariance) If $A\ind_CB$ and $(A',B',C')\equiv (A,B,C)$, then $A'\ind_{C'}B'$.
\item (Monotonicity) If $A\ind_C B$, $A'\subseteq A$, and $B'\subseteq B$, then $A'\ind_C B'$.
\item (Base monotonicity) Suppose $C\in[D, B]$.  If $A\ind_D B$, then $A\ind_C B$.
\item (Transitivity) Suppose $C\in[D, B]$.  If $B\ind_C A$ and $C\ind_D A$, then $B\ind_D A$.
\item (Normality) $A\ind_CB$ implies $AC\ind_C B$.
\item (Extension) If $A\ind_C B$ and $\hat B\supseteq B$, then there is $A'\equiv_{BC} A$ such that $A'\ind_C \hat B$.
\item (Finite character) If $A_0\ind_CB$ for all finite $A_0\subseteq A$, then $A\ind_CB$.
\item (Local character) For every $A$, there is a cardinal $\kappa(A)<\upsilon$ such that, for any set $B$, there is a subset $C$ of $B$ with $|C|<\kappa(A)$ such that $A\ind_CB$.
\end{enumerate}
If finite character is replaced by countable character (which is defined in the obvious way), then we say that $\ind$ is a
\emph{countable independence relation}.  We will refer to the first five axioms (1)--(5) as the \emph{basic axioms}.
\end{df}

As the trivial independence relation (which declares $A\ind_CB$ to always hold) is obviously of little interest, one adds an extra
condition to avoid such trivialities.

\begin{df}
An independence relation $\ind$ is \emph{strict} if it satisfies
\begin{itemize}
\item[(9)] (Anti-reflexivity) $a\ind_B a$ implies $a\in \acl(B)$.
\end{itemize}
\end{df}

There are two other useful properties to consider when studying ternary relations over $\cu N$:

\begin{df}

\

\begin{enumerate}
\item[(10)] (Full existence) For every $A,B,C$, there is $A'\equiv_C A$ such that $A'\ind_C B$.
\item[(11)] (Symmetry) For every $A,B,C$, $A\ind_CB$ implies $B\ind_CA$.
\end{enumerate}
\end{df}

\begin{rmks}  \label{r-fe-vs-ext}

\noindent\begin{enumerate}
\item  Whenever $\ind$ satisfies invariance, monotonicity, transitivity, normality, full existence, and symmetry,
then $\ind$ also satisfies extension (Remark 1.2 in Ad2]).
\item If $\ind$ satisfies base monotonicity and local character, then $A\ind_C C$ for all small $A,C$ (Appendix to [Ad1]).
\item If $\ind$ satisfies monotonicity and extension, and $A\ind_C C$ holds for all small $A,C$, then $\ind$ also satisfies full existence
(Appendix to [Ad1]).
\item  Any countable independence relation is symmetric.
\end{enumerate}

\noindent While the proofs of these results in [Ad1] and [Ad2] are in
the first order setting, it is straightforward to check that they persist in the continuous setting.  Theorem 2.5 in [Ad2] shows that any independence
relation is symmetric.  The same argument with Morley sequences of length $\omega_1$ instead of countable Morley sequences proves (4).
\end{rmks}

\begin{rmk}  \label{r-localchar}  (Compare with Remark 1.3 in [Ad2])
If $\ind$ has invariance, countable character, base monotonicity, and satisfies local character when $A$ is countable, then $\ind$ has local character.
\end{rmk}

\begin{proof}
Fix a small set $D$.  By hypothesis, for each countable subset $A$ of $D$ there is a smallest cardinal $\lambda(A)<\upsilon$ such that,
for any small set $B$, there is a subset $C(A,B)$ of $B$ with $|C|\le\lambda(A)$ such that $A\ind_{C(A,B)}B$.
By invariance, whenever $E\equiv A$ we have $\lambda(E)=\lambda(A)$.  Since there are countably many formulas, there are at most $2^{\aleph_0}$
different values of $\lambda(A)$ for $A$ countable.
Let $\lambda$ be the sum of $\lambda(A)$ over all countable subsets $A$ of $D$, and let $\kappa(D)=\lambda^+$.  Since $\upsilon$ is uncountable inaccessible, $\kappa(D)<\upsilon$.
Now fix a small set $B$, and let $C=\bigcup\{C(A,B)\colon A\subseteq D, |A|\le\aleph_0\}$.  Then $|C|\le \lambda$.
By base monotonicity, we have $A\ind_C B$ for each countable $A\subseteq D$.
By countable character, $D\ind_C B$, so $\ind$ has local character with bound $\kappa(D)$.
\end{proof}

\begin{df}
We say that $\ind$ has \emph{countably local character} if for every countable set $A$ and every set $B$, there is a countable subset $C$ of $B$ such that $A\ind_CB$.
\end{df}

Note that countably local character implies local character when $A$ is countable (with $\kappa(D)=\aleph_1$).

\begin{rmk}  \label{r-countably-localchar}
\noindent\begin{enumerate}
\item If $\ind$ has local character with bound $\kappa(D)=(|D|+\aleph_0)^+$, then $\ind$ has countably local character.
\item If $\ind$ has invariance, countable character, base monotonicity, and countably local character, then
$\ind$ has local character with bound $\kappa(D)=((|D| + 2)^{\aleph_0})^+$.
\end{enumerate}
\end{rmk}

\begin{proof}  (1) is obvious.  (2) follows from the proof of Remark \ref{r-localchar}, with $\lambda(A)\le\aleph_0$ for each countable $A$, and $\lambda=(|D|+2)^{\aleph_0}$.
\end{proof}

We say that $\ind[J]$ is \emph{weaker than} $\ind[I]$, and write $\ind[I]\Rightarrow \ind[J]$, if $A\ind[I]_C B\Rightarrow A\ind[J]_C B$.

\begin{rmk}  \label{r-weaker}  Suppose $\ind[I]\Rightarrow\ind[J]$.  If $\ind[I]$ has full existence, local character,
or countably local character, then $\ind[J]$ \ has the same property.
\end{rmk}

\subsection{Special Independence Relations}

\noindent In any complete theory (first order or continuous), we define the notion of \emph{algebraic independence}, denoted $\ind[a]$,
by setting $A\ind[a]_CB$ to mean $\acl(AC)\cap \acl(BC)=\acl(C)$. In first order logic, $\ind[a]$ satisfies all axioms for a strict
independence relation except for perhaps base monotonicity.

\begin{prop}  \label{p-alg-indep}
In any complete continuous theory, $\ind[a]$ satisfies symmetry and all axioms for a strict countable independence relation except perhaps for
base monotonicity and extension.
\end{prop}

\begin{proof}
The proof is exactly as in [Ad2], Proposition 1.5, except for some minor modifications.  For example, countable character of $\acl$ in
continuous logic yields countable character of $\ind[a]$.  Also, in the verification of local character, one needs to take
$\kappa(A):=((|A|+2)^{\aleph_0})^+$ instead of $(|A|+\aleph_0)^+$.
\end{proof}

\begin{question}  \label{q-a-full-existence}
Does $\ind[a]$ always have full existence (or extension) in continuous logic?
\end{question}

The lattice of algebraically closed sets is \emph{modular} if
$$ B\cap(\acl(AC))=\acl((B\cap A)C)$$
 whenever $A, B, C$ are algebraically closed and $C\subseteq B$.
Proposition 1.5  in [Ad2] shows that in first order logic, $\ind[a]$ satisfies base monotonicity, and is thus a strict independence relation,
if and only if the lattice of algebraically closed sets is modular.  In continuous logic, the same argument still shows that $\ind[a]$
satisfies base monotonicity if and only if the lattice of algebraically closed sets is modular.

Recall the following definitions from [Ad2]:

\begin{df}  \label{d-special-ind}
Suppose that $A,B,C$ are small subsets of the big model of a complete theory.
\begin{itemize}
\item $A\ind[M]_CB$ iff for every $D\in[C,\acl(BC)]$, we have $A\ind[a]_{D}B$.
\item $A\thind_CB$ iff for every (small) $E\supseteq BC$, there is $A'\equiv_{BC}A$ such that $A'\ind[M]_CE$.
\end{itemize}
\end{df}


Note that $\thind\Rightarrow\ind[M] \ $ and $\ind[M]\Rightarrow\ind[a]$.  Also, $\thind = \ind[M] \ \ $ if and only if $\ind[M] \ $ satisfies extension.
In [Ad2], it is shown that, in the first order setting, $\ind[M]$ \ satisfies all of the axioms for a strict independence relation except for
perhaps local character and extension (but now base monotonicity has been ensured).  It is shown in [EG] that
this fact remains true in continuous logic (except, of course, for finite character being replaced by countable character).

By [Ad1] and [Ad2], in the first order setting
$\thind$ satisfies all of the axioms for a strict independence relation except perhaps  local character
\footnote{Finite character follows from Proposition A.2 of [Ad1], and is stated explicitly in [Ad3], Proposition 1.3.}
(but now extension has been ensured).  It is shown in [EG] that in the continuous setting,
$\thind$ satisfies all of the axioms for a strict countable independence relation except perhaps for countable character and local character.

A (continuous or first order) theory $T$ is said to be \emph{real rosy} if $\thind$ has local character.
Examples of real rosy theories are the first order or continuous stable theories ([BBHU, Section 14), the first order or continuous simple theories (see [Be3]),
and the first order o-minimal theories (see [On]).

The following results are consequences of Theorem 3.2 and the preceding discussion in [EG].  The proof of part (2) is the same as the proof of Remark
4.1 in [Ad2].

\begin{result}  \label{f-weakest}

\noindent\begin{enumerate}
\item $T$ is real rosy if and only if $\thind$ is a strict countable independence relation.
\item
If $\ind[I]$ satisfies the basic axioms and extension, and is symmetric and anti-reflexive, then $\ind[I]\Rightarrow\thind$.
\item If a theory has a strict countable independence relation, then it is real rosy, and $\thind$ is the weakest strict countable
independence relation.
\end{enumerate}
\end{result}

The next result is a consequence of Theorem 8.10 in [BU] and Theorem 14.18 in [BBHU].

\begin{result}  \label{f-stable-indep}
If $T$ is stable, then it has a unique strict independence relation,
the forking independence relation $\ind[f]$. Moreover, $\ind[f]$ has countably local character.
\end{result}

\begin{cor}  \label{c-stable-thorn}
If $T$ stable and $\thind$ has finite character, then $\ind[f]=\thind$.
\end{cor}

\begin{proof}
By Results \ref{f-weakest} and \ref{f-stable-indep}.
\end{proof}

A continuous formula $\Phi(\vec x,B,C)$  \emph{divides over} $C$ if, in the big model of $T$,
there is a $C$-indiscernible sequence $\<B^i\>_{i\in\BN}$
such that $B^0\equiv_C B$ and the set of statements $\{\Phi(\vec x,B^i,C)=0\colon i\in\BN\}$ is not satisfiable.
The \emph{dividing independence relation} $A\ind[d]_C B$ is defined to hold if there is no tuple $\vec a\in A^{<\BN}$ and continuous formula
$\Phi(\vec x,B,C)$ such that $\Phi(\vec a,B,C)=0$ and $\Phi(\vec x, B,C)$ divides over $C$.

\begin{result}  \label{f-d-indep}

\noindent\begin{enumerate}
\item $\ind[d]$ satisfies invariance, monotonicity, and finite character.
\item If $T$ is stable then $\ind[d]$ is equal to the unique strict independence relation $\ind[f]$.
\end{enumerate}
\end{result}

\begin{proof}   Invariance and monotonicity are clear, and finite character holds because
each formula has only finitely many variables and parameters.  Part (2) follows from Theorem 14.18 in [BBHU].
\end{proof}

\section{Countably Based Relations}

In this section we will introduce the notion of a countably based ternary relation over $\cu N$.  The reason this notion is useful is because for
each ternary relation with monotonicity there is a unique countably based ternary relation that agrees with it on countable sets (Lemma \ref{l-cbased}).
This will be important in Section \ref{s-pointwise}, where it allows us to introduce, for each ternary relation with monotonicity over the big model of a first order theory $T$,
a corresponding ``pointwise independence relation'' over the big model of $T^R$ (Definition \ref{d-measurable}).

The notions and results in this section hold for all first order and continuous theories.  We will give the proofs only for continuous
theories; the proofs for first order theories are similar but simpler.
We will use $(\forall^c D)$ to mean ``for all countable $D$'', and similarly for $(\exists^c D)$.

\subsection{The General Case}

\begin{df}
We say that a ternary relation $\ind[I]$ is \emph{countably based} if for all $A, B, C$ we have
$$ A\ind[I]_C B \Leftrightarrow (\forall^c A'\subseteq A)(\forall^c B'\subseteq B)(\forall^c C'\subseteq C)(\exists^c D\in[C',C])A'\ind[I]_{D} B'.$$
\end{df}

Note that if $\ind[I]$ and $\ind[J]$ are countably based and agree on countable sets, then they are the same.

\begin{df}  We say that $\ind[I]$ has \emph{two-sided countable character} if $\ind[I]$ has countable character and
$$ [(\forall^c B_0\subseteq B)A\ind[I]_C B_0] \Rightarrow A\ind[I]_C B.$$
In other words,
$$ [(\forall^c A_0\subseteq A)(\forall^c B_0\subseteq B) A_0\ind[I]_C B_0] \Rightarrow A\ind[I]_C B.$$
\end{df}

\begin{rmk} \label{r-two-sided}
If $\ind[I]$ has symmetry and countable character, then $\ind[I]$ has two-sided countable character.
\end{rmk}

\begin{proof}  Suppose $(\forall^c B_0\subseteq B)A\ind[I]_C B_0$.  By symmetry, $(\forall^c B_0\subseteq B)B_0\ind[I]_C A$.
By countable character, $B\ind[I]_C A$.  Then by symmetry again, $A\ind[I]_C B$.
\end{proof}

\begin{lemma}  \label{l-cbased}

\noindent\begin{enumerate}
\item Suppose $\ind[I]$ and $\ind[J]$ are countably based.  If
$$A\ind[I]_C B \Rightarrow A\ind[J]_C B$$
holds for all countable $A,B,C$, then it holds for all small $A,B,C$.
\item  Suppose $\ind[I]$ has monotonicity.  There is a unique ternary relation $\ind[Ic]$ \  that is countably
based and agrees with $\ind[I]$ on countable sets.  Namely,
$$ A\ind[Ic]_C\,\, B \Leftrightarrow (\forall^c A'\subseteq A)(\forall^c B'\subseteq B)(\forall^c C'\subseteq C)(\exists^c D\in[C',C])A'\ind[I]_{D} B'.$$
\item $\ind[I]$ is countably based if and only if $\ind[I]$ has monotonicity and two-sided countable character, and
whenever $A$ and $B$ are countable, we have
$$A\ind[I]_C B \Leftrightarrow (\forall^c C'\subseteq C)(\exists^c D\in[C',C]) A\ind[I]_{D} B.$$
\item Suppose $\ind[I]$ is countably based.  If $\ind[I]$ has invariance, base monotonicity, transitivity, normality, or symmetry
for all countable $A,B,C$, then $\ind[I]$ \ \ has the same property for all small $A,B,C$.
\end{enumerate}
\end{lemma}

\begin{proof}  (1) follows easily from the definition of countably based.

(2) Uniqueness is clear.  Let $\ind[J]$\, be the relation defined by the displayed formula and let $A,B,C$ be countable.
It is obvious that $\ind[J]$ is countably based, and that $A\ind[J]_C B$ implies $A\ind[I]_C B$.  Suppose $A\ind[I]_C B$
and $A'\subseteq A, B'\subseteq B, C'\subseteq C$.  By monotonicity for $\ind[I]$ we have $A'\ind[I]_C B'$.  But $C\in[C',C]$, so $A\ind[J]_C B$
as required.

(3) Suppose first that $\ind[I]$ is countably based.  It is immediate that $\ind[I]$ has monotonicity and two-sided countable character.
Then $\ind[I]=\ind[Ic]$ \ \ by (1).  Therefore the following statements are equivalent:
\begin{itemize}
\item[(a)] $A\ind[I]_C B$;
\item[(b)] $(\forall^c A'\subseteq A)(\forall^c B'\subseteq B) A'\ind[I]_C B'$;
\item[(c)] $(\forall^c A'\subseteq A)(\forall^c B'\subseteq B) A'\ind[Ic]_C \ B'$;
\item[(d)] $(\forall^c A'\subseteq A)(\forall^c B'\subseteq B)(\forall^c C'\subseteq C)(\exists^c D\in[C',C]) A'\ind[I]_D B'$;
\end{itemize}
When $A$ and $B$ are countable, (d) says that
$$ (\forall^c C'\subseteq C)(\exists^c D\in[C',C]) A\ind[I]_D B,$$
so the displayed formula in (2) holds.

Now suppose that $\ind[I]$ has monotonicity and two-sided countable character, and the displayed formula in (2) holds whenever
$A, B$ are countable.   Then the statements (a)--(d) are again equivalent, so $\ind[I]$ is countably based.

(4)  We prove the result for base monotonicity.  Suppose $\ind[I]$ has base monotonicity for countable sets, and assume that $A\ind[I]_D \ B$ and $C\in[D,B]$.
We prove $A\ind[I]_C \ B$.  Let $A'\subseteq A, B'\subseteq B, C'\subseteq C$ be countable.  Since $\ind[I]$ is countably based,
it suffices to find a countable
$C''\in[C',C]$ such that $A'\ind[I]_{C''} B'$.  Let $B''=B'\cup C'$ and $D':=C'\cap D$.  Note that $B''$ is a countable subset of $B$ and
$C'\in[D',B'']$.   Since $A\ind[I]_D \ B$ there is a countable $E\in[D',D]$ such that $A'\ind[I]_E B''$.
Let $C''=E\cup C'$.  Then $C''\in [C',C]$ and $C''\in [E,B'']$, so by base monotonicity for countable sets we have $A'\ind[I]_{C''} B''$.
Then by monotonicity, we have $A'\ind[I]_{C''} B'$ as required.
\end{proof}

\begin{prop}  \label{p-preserve}
Let $\ind[I]$ be a ternary relation that has monotonicity.
If $\ind[I]$ has any of invariance, base monotonicity, transitivity, normality, symmetry, or anti-reflexivity for all countable sets,
then $\ind[Ic]$ \  has the same property for all small sets.
\end{prop}

\begin{proof}  Invariance, and symmetry are clear.

Base monotonicity:  Suppose $C\in[D,B]$ and $A\ind[Ic]_D \ B$.  Let $A_0\subseteq A, B_0\subseteq B$, ${C_0}\subseteq C$ be countable.
Let $D_0=C_0\cap D$.  Then there exists a countable $D_1\in[D_0,D]$ such that $A_0\ind[I]_{D_1} B_0$.  Let $C_1=C_0\cup D_1$.
By base monotonicity for $\ind[I]$, $A_0\ind[I]_{C_1} B_0$.  Therefore $A\ind[Ic]_C \ B$.

Transitivity:  Assume $C\in[D,B], B\ind[Ic]_C \ A$, and $C\ind[Ic]_D \ A$.  Let $A_0\subseteq A, B_0\subseteq B$, ${C_0}\subseteq C$,
$D_0\subseteq D$  be countable.  There is a countable $C_1\in[C_0,C]$ such that $B_0\ind[I]_{C_1} A_0$, and a countable $D_1\in [D_0,D]$
such that $C_1\ind[I]_{D_1} A_0$.  By transitivity of $\ind[I]$, $B_0\ind[I]_{D_1} A_0$.  This shows that $B\ind[Ic]_D \ A$.

Normality:  Assume $A\ind[Ic]_C \ B$.  Let $E_0\subseteq AC, B_0\subseteq B$, $C_0\subseteq C$ be countable.  Let $A_0=E_0\cap A, C=E_0\cap C$.
Then for some countable $C_1\in[C_0,C]$ we have $A_0\ind[I]_{C_1} B_0$.  By normality of $\ind[I]$,  $A_0C_1\ind[I]_{C_1} B_0$.
By monotonicity of $\ind[I]$, $E_0\ind[I]_{C_1} B_0$.  Thus  $AC\ind[Ic]_C \ B$.

Anti-reflexivity:  Suppose $a\ind[Ic]_C \ a$.  Let $C_0\subseteq C$ be countable.  For some countable $C_1\in[C_0,C]$ we have
$ a\ind[I]_{C_1}  a$.  Then $ a\in\acl(C_1)$ by the anti-reflexivity of $\ind[I]$, so $ a\in\acl(C)$.
\end{proof}

The following property is sometimes useful in proving that a relation has finite character or is countably based.

\begin{df} A ternary relation $\ind[I]$ has the \emph{countable union property} if whenever
$A, B, C$ are countable, $C=\bigcup_n C_n$, and $C_n\subseteq C_{n+1}$ and $A\ind[I]_{C_n} B$
for each $n$, we have $A\ind[I]_C B$.
\end{df}

Given two ternary relations $\ind[I]$ and $\ind[J]$ over $\cu N$, $\ind[I]\wedge\ind[J]$ will denote the relation $\ind[K] \ $ such that
$$A\ind[K]_C \ B\Leftrightarrow A\ind[I]_C B\wedge A\ind[J]_C B.$$

\begin{prop}  \label{p-union-pair}  Suppose $\ind[I]$ and $\ind[J]$ are both countably based and have the countable union property.
Then the relation $\ind[I]\wedge\ind[J]$ is also countably based.
\end{prop}

\begin{proof}   Let $A, B$ be countable and let $\ind[K] \ =\ind[I]\wedge\ind[J]$.  By  Lemma \ref{l-cbased} (3), it is enough to show that
$$ A\ind[K]_C \ B \Leftrightarrow (\forall^c C'\subseteq C)(\exists^c D\in[C',C]) A\ind[K]_D \ B.$$
The implication from right to left is trivial. For the other direction, assume $A\ind[K]_C B$ and let $C'\subseteq C$ be countable.
Since both $\ind[I]$ and $\ind[J]$ are countably based, there is a  sequence $\< D_n\>_{n\in\BN}$ of countable sets such that
$D_n\subseteq D_{n+1}$ and $D_n\in[C',C]$ for each $n\in\BN$, $A\ind[I]_{D_n} B$ for each even $n$, and  $A\ind[J]_{D_n} B$ for each odd $n$.
Let $D=\bigcup_n D_n$. Then $D\in[C',C]$ and $D$ is countable.  Since both $\ind[I]$ and $\ind[J]$ have the countable union property, we have $A\ind[K]_D B$,
as required.
\end{proof}

\begin{prop} \label{p-preserve-finitechar}  If $\ind[I]$ has monotonicity,
finite character, and the countable union property, then $\ind[Ic] \ $ has finite character.
\end{prop}

\begin{proof}
Suppose $A'\ind[Ic]_C \ B$ for every finite $A'\subseteq A$.  Let
$A_0\subseteq A, B_0\subseteq B$, ${C_0}\subseteq C$ be countable.  Let $A_0=\bigcup_{n}E_n$ where
$E_n$ is finite and $E_n\subseteq E_{n+1}$ for each $n$.  By induction on $n$, there is a sequence
of countable sets $\<D_n\>_{n\in\BN}$ such that for each $n$, $D_n\in[C_0,C]$, $D_n\subseteq D_{n+1}$, and $E_n\ind[I]_{D_n} B_0$.
By monotonicity, $E_n\ind[I]_{D_k} B_0$ whenever $n\le k$.  Let $D=\bigcup_n D_n$.  Then $D$ is countable and $D\in[C_0,C]$.
By the countable union property, $E_n\ind[I]_D B_0$ for each $n$.  Hence by monotonicity and finite character for $\ind[I]$,
we have $A_0\ind[I]_D B_0$.  This shows that $A\ind[Ic]_C \ B$, so $\ind[Ic] \ $ has finite character.
\end{proof}

\begin{prop}  \label{p-countably-local-implies}
Suppose $\ind[I]$ has  monotonicity, base monotonicity, transitivity, symmetry, and countably local character.  Then $\ind[I]\Rightarrow\ind[Ic] \ $.
\end{prop}

\begin{proof}  Suppose  $A\ind[I]_C B$.  Let $A'\subseteq A, B'\subseteq B, C'\subseteq C$ be countable.  By monotonicity, $A'\ind[I]_C B'$.
Countably local character insures that
there is a countable $C_1\subseteq C$ such that $A'\ind[I]_{C_1} C$.  Let $D=C_1 C'$.  Then $D$ is countable and $D\in[C_1,C]$.
By base monotonicity, $A'\ind[I]_{D} C$.  By symmetry, $B'\ind[I]_C A'$ and $C\ind[I]_{D} A'$. By transitivity, $B'\ind[I]_D A'$, and by
symmetry again, $A'\ind[I]_D B'$.  Moreover, $D\in[C',C]$.  This proves that $A\ind[Ic]_C \ B$.
\end{proof}

\begin{cor} \label{c-preserve}
Let $\ind[I]$ be a countable independence relation.

\begin{enumerate}
\item If $\ind[I]$ has countably local character, then $\ind[I]\Rightarrow\ind[Ic] \ \ $.
\item If $\ind[I]\Rightarrow\ind[Ic] \ \ $ then $\ind[Ic] \ \ $ is a countable independence relation.
\end{enumerate}
\end{cor}

\begin{proof} (1): $\ind[I]$ has symmetry by Remarks \ref{r-fe-vs-ext} (4).  (1) follows from symmetry and Proposition \ref{p-countably-local-implies}.

(2): $\ind[Ic]$ \ \ has monotonicity and countable character by Lemma \ref{l-cbased}.  By Remark \ref{r-weaker} and Proposition \ref{p-preserve},
$\ind[Ic]$ \ \ satisfies full existence, symmetry, and all the axioms except perhaps extension.  By Remarks \ref{r-fe-vs-ext}, extension follows
from full existence, symmetry, and the other axioms, so $\ind[Ic]$ \ \ satisfies extension as well.
\end{proof}

\subsection{Special Cases}

We will show that in continuous logic (as well as first order logic),  $\ind[a]$ and $\ind[M]$ \ are countably based.
We also give conditions under which $\ind[d]$ and $\thind$  are countably based.

\begin{prop}  \label{p-a-cbased}   The relation $\ind[a]$ is countably based.
\end{prop}

\begin{proof}  $\ind[a]$ has two-sided countable character and monotonicity.  Let $A$ and $B$ be countable.
By Lemma \ref{l-cbased} (3), it is enough to show that
$$ A\ind[a]_C B \Leftrightarrow (\forall^c C'\subseteq C)(\exists^c D\in[C',C]) A\ind[a]_{D} B.$$
Suppose that $ A\ind[a]_C B$, and let $C'\subseteq C$ be countable.
The function $d(\cdot,C)$ is uniformly continuous, so $(\cu N,d(\cdot,C))$ is a structure.
By Fact \ref{f-definableclosure}, $\acl(ABC')$ is separable.  By the L\"owenheim-Skolem theorem, there is a separable elementary
substructure  $(\cu P,d(\cdot,C_0))\prec(\cu N,d(\cdot,C))$ such that $\acl(ABC')\subseteq\cu P$.  Note that $C'\subseteq C_0=C\cap \cu P$.
There is a countable set $D\in[C',C_0]$ such that $D$ is dense in $\cl(C_0)$. Then $\acl(C_0)=\acl(D)$.
In $\cu N$ we have $\acl(AC)\cap\acl(BC)\subseteq \acl(C)$, so in $\cu P$ we have
$$\acl(AD)\cap\acl(BD)= \acl(AC_0)\cap\acl(BC_0)\subseteq\acl(C_0)=\acl(D),$$
and hence $A\ind[a]_D B$.

For the other direction, suppose that
$$ (\forall^c C'\subseteq C)(\exists^c D\in[C',C]) A\ind[a]_{D} B.$$
Let $c\in\acl(AC)\cap\acl(BC)$.  By Fact \ref{f-algebraic}, there is a countable $C'\subseteq C$ such that
$c\in \acl(AC')\cap\acl(BC')$.  Take a countable $D\in[C',C]$ with $A\ind[a]_D B$.  Then
$$c\in\acl(AD)\cap\acl(BD)=\acl(D)\subseteq\acl(C),$$
so $A\ind[a]_C B$.
\end{proof}

\begin{lemma} \label{l-2sided}
The relation $\ind[M]$ \ has two-sided countable character.
\end{lemma}

\begin{proof}
Suppose that $A'\ind[M]_C B'$ for every countable $A'\subseteq A$ and $B'\subseteq B$.  We will show that $A\ind[M]_C B$.
Let $D\in[C,\acl(BC)]$ and $x\in\acl(AD)\cap\acl(BD)$.  We must prove that $x\in\acl(D)$.
There are countable subsets $A_0\subseteq A$ $B_0\subseteq B$, and $D_0\subseteq D$ such that $x\in\acl(A_0D_0)\cap\acl(B_0D_0)$.
There is a countable set $B_1\in[B_0,B]$ such that $D_0\subseteq\acl(B_1C)$.  Let $D_1=D_0\cup C$.  Then
$D_1\in[C,\acl(B_1C)]$ and $D_1\subseteq D$.
We have $A_0\ind[M]_C B_1$, so $A_0\ind[a]_{D_1} B_1$.  Moreover, $x\in\acl(A_0D_1)\cap\acl(B_1D_1)$, so $x\in \acl(D_1)\subseteq\acl(D)$.
\end{proof}

\begin{prop} \label{p-M-cbased}
The relation $\ind[M]$ \ is countably based.
\end{prop}

\begin{proof}
$\ind[M]$ \ has monotonicity.
By Lemma \ref{l-2sided},  $\ind[M]$ \ has two-sided countable character.  Let $A$ and $B$ be countable.
By Lemma \ref{l-cbased} (3), it is enough to show that
$$ A\ind[M]_C \ B \Leftrightarrow (\forall^c C'\subseteq C)(\exists^c D\in[C',C]) A\ind[M]_{D} \ B.$$
Suppose $A\ind[M]_C B$ and let $C'\subseteq C$ be countable. As before, we let $(\cu P,d(\cdot,C_1))$ be a separable elementary
substructure of $(\cu N,d(\cdot,C))$ such that $\acl(ABC')\subseteq\cu P$, and take a countable set $D\in[C',C_1]$ such that $D$ is dense in $\cl(C_1)$.
In $\cu N$ we have $A\ind[a]_F B$ for every $F\in[C,\acl(BC)]$.  Let $G\in [D,\acl(BD)]$, and suppose
$x\in\acl(AG)\cap\acl(BG)$.  Let $F=CG$.  Then $F\in[C,\acl(BC)]$, so $A\ind[a]_F B$ in $\cu N$ and $x\in\acl(AF)\cap\acl(BF)$.
Therefore $x\in\acl(F)$.  Using the definition of algebraic closure, it follows that in $\cu P$, $x\in\acl(GD)=\acl(G)$.
This shows that $A\ind[a]_G B$, so $A\ind[M]_D B$.

For the other direction, suppose that
$$ (\forall^c C'\subseteq C)(\exists^c D\in[C',C]) A\ind[M]_{D} \ B.$$
Let $E\in[C,\acl(BC)]$ and let $c\in\acl(AE)\cap\acl(BE)$.
By Fact \ref{f-algebraic}, there is a countable $E'\subseteq E$ such that
$c\in \acl(AE')\cap\acl(BE')$.  There is also a countable $C'\subseteq C$ such that $E'\subseteq\acl(BC')$.
Take a countable $D\in[C',C]$ with $A\ind[M]_D B$.  Let $D'=D\cup E'$.  Then
$D'\in[D,\acl(BD)]$, so $A\ind[a]_{D'} B$.  Moreover, $c\in\acl(AD')\cap\acl(BD')$, so $c\in\acl(D')$.
Finally, $D'\subseteq C\cup E=E$, so $c\in\acl(E)$.  This proves that $A\ind[M]_C B$.
\end{proof}

\begin{prop}  \label{p-thorn-cbased}
If there exists a strict countably based independence relation $\ind[I]$ over $\cu N$ with countably local character, then $\thind$ is countably based.
\end{prop}

\begin{proof}  By Result \ref{f-weakest}, $\thind$ is the weakest strict countable independence relation on models of $T$.
Then $\ind[I] \Rightarrow \thind$, so $\thind$ has countably local character.
By Corollary \ref{c-preserve}, $\ind[\th c]$ \ \ is a strict countable independence relation on models of $T$ that is
weaker than $\thind$.  Therefore $\ind[\th c] \ =\thind$, so $\thind$ is countably based.
\end{proof}

\begin{prop} \label{p-d-cbased}  Suppose the dividing independence relation $\ind[d]$ over $\cu N$
is an independence relation with countably local character.  Then $\ind[d]$ is countably based.
\end{prop}

\begin{proof}
By Lemma \ref{l-cbased} it is enough to check that, for countable $A$, $B$ and small $C$, we have
\begin{equation}  \label{e-d-cbased}
A\ind[d]_CB \Leftrightarrow (\forall^c C'\subseteq C)(\exists^c D\in [C',C])A\ind[d]_DB.
\end{equation}
Fix such $A$, $B$, $C$.

$\Rightarrow$:  Suppose that $A\ind[d]_CB$.  Since $\ind[d]$ is an independence relation with countably local character,
we have that $\ind[d]\Rightarrow \ind[dc] \ \ $ by Corollary \ref{c-preserve}, whence we get the forward implication of (\ref{e-d-cbased}).

$\Leftarrow$:  Suppose that $A\nind[d]_CB$.  Then for some $\vec a\in A^{<\BN}$ and some continuous formula $\Phi(\vec x,B,C)$,
$\cu N\models \Phi(\vec a,B,C)=0$ and $\Phi(\vec x,B,C)$ divides over $C$.  Take a countable (even finite) $C'\subseteq C$ such
that $\Phi(\vec x,B,C)=\Phi(\vec x,B ,C')$.   Then for any countable $D\in [C',C]$, $\Phi(\vec x,B,C)$ divides over $D$,
so $A\nind[d]_D B$ and the right hand side of (\ref{e-d-cbased}) fails.
\end{proof}

In the paper [Be3],  Ben Yaacov defined simple continuous theories and showed that they satisfy the hypotheses of
Proposition \ref{p-d-cbased}.  Thus on models of a simple theory, $\ind[d]$ is countably based.

\begin{lemma}  \label{l-d-union}
The relation $\ind[d]$ has the countable union property.
\end{lemma}

\begin{proof}
Suppose $A, B, C$ are countable, $C=\bigcup_n C_n$, and $C_n\subseteq C_{n+1}$ and $A\ind[d]_{C_n} B$ for each $n$,
but $A\nind[d]_C B$.  Then there exists $\vec a\in A^{<\BN}$ and a continuous formula $\Phi(\vec x,B,C)$ such that
$\Phi(\vec a,B,C)=0$ and $\Phi(\vec x, B,C)$ divides over $C$.  Then $\Phi(\vec x,B,C)=\Phi(\vec x,B,C_n)$ for some $n\in\BN$.
Hence $\Phi(\vec a,B,C_n)=0$ and $\Phi(\vec x,B,C_n)$ divides over $C_n$, contradicting $A\ind[d]_{C_n} B$.
\end{proof}

\begin{prop}  \label{p-f-cbased}  On models of a stable theory, $\thind$ and $\ind[f]$ are countably based.
\end{prop}

\begin{proof}  By Theorem 8.10 in [BU], $\ind[f]$ is the unique strict independence relation over models of $T$, and has countably local character.
So by Proposition \ref{p-thorn-cbased}, $\thind$ is countably based.

By Corollary \ref{c-preserve}, $\ind[fc] \ $ is a countable independence relation.
By Proposition  \ref{p-preserve}, $\ind[fc] \ $ is anti-reflexive.
By Result \ref{f-d-indep}, $\ind[d]=\ind[f]$, so by Lemma \ref{l-d-union},  $\ind[f]$ has the countable union property.
Then by  Proposition \ref{p-preserve-finitechar}, $\ind[fc] \ $ has finite character.  Hence by the uniqueness of $\ind[f]$,
$\ind[fc] \ = \ind[f]$, so $\ind[f]$ is countably based.
\end{proof}

Our next result concerns theories $T$ in \emph{first order} logic that are real rosy.
Let $\vec b$ be a tuple and $C$ be a small set in a big model $\cu M$ of a first order theory $T$.
By Definition 2.1 in Onshuus [On], a first order formula $\psi(\vec x,\vec b)$ \emph{$\th$-divides over} $C$
if there is $k\in\BN$ and a finite tuple $\vec e$ such that $\{\psi(\vec x,\vec b')\colon \vec b'\equiv_{C\vec e}\vec b\}$
is $k$-inconsistent, and $\vec b$ is not contained in $\acl(C\vec e)$.  A formula $\varphi(\vec x,\vec b)$ \emph{\th-forks over}
$C$ if it implies a finite disjunction of formulas that $\th$-divide over $C$.  We use the following result from [Ad1].

\begin{result}  \label{f-thorn-forks}  ([Ad1], Proposition A.2)
For small sets $A,B,C$ in a big model of a first order theory $T$, the following are equivalent:

\begin{itemize}
\item $A\thind_C B$.
\item $A\thind_C B$ in the sense of [On], that is, for every tuple $\vec a\in A^{<\BN}$ and $\vec b\in(BC)^{<\BN}$, there
is no formula $\varphi(\vec x,\vec y)$ such that $\cu M\models\varphi(\vec a,\vec b)$ and $\varphi(\vec x,\vec b)$
$\th$-forks over $C$.
\end{itemize}
\end{result}


\begin{rmk}  \label{r-th-two-sided}
It follows at once that in the first order setting, $\thind$ has two-sided finite character.
\end{rmk}

\begin{lemma}  \label{l-thorn-union}  For every complete first order theory $T$, the relation $\thind$ has the countable union property.
\end{lemma}

\begin{proof}  We first suppose that $A, B$ are countable sets and $C$ is a small set in the big model $\cu M$ of $T$, and that $A\nthind_CB$.
By Result \ref{f-thorn-forks} there is a formula $\varphi(\vec x)\in \tp(A/BC)$ that $\th$-forks over $C$.  By definition, this means that
there is an $m\in\BN$ and formulas $\psi_i(\vec x,\vec d_i), i<m$, such that
$\varphi(\vec x)\vdash \bigvee_{i<m} \psi_i(\vec x,\vec d_i)$, and each $\psi_i(\vec x,\vec d_i)$
$\th$-divides over $C$.  This in turn means that for each $i<m$, there is a finite tuple $\vec e_i$ and $k_i\in\BN$ such that
$\{\psi_i(\vec x,\vec d_i')  :  \vec d_i'\equiv_{C\vec e_i}\vec d_i\}$ is $k_i$-inconsistent and $\vec d_i\notin \acl(C\vec e_i)$.
There is a finite $C_0\subseteq C$ such that $\varphi(\vec x)$ is an $L(BC_0)$-formula.  By compactness, there is a finite $C_1\in[C_0,C]$
such that, for each $i<m$, $\{\psi_i(\vec x,\vec d_i')  :  \vec d_i'\equiv_{C_1\vec e_i}\vec d_i\}$ is $k_i$-inconsistent.

Now suppose that $C=\bigcup_n D_n$, and $D_n\subseteq D_{n+1}$ for each $n$.  Then for some $n$ we have $C_1\subseteq D_n$, and hence
for each $i<m$, $\{\psi_i(\vec x,\vec d_i')  :  \vec d_i'\equiv_{D_n\vec e_i}\vec d_i\}$ is $k_i$-inconsistent and $\vec d_i\notin \acl(D_n\vec e_i)$.
Therefore each $\psi_i(\vec x,\vec d_i)$ $\th$-divides over $D_n$, so $A\nthind_{D_n} B$.  This shows that $\thind$ has the countable union property
(even for $C$ small instead of countable).
\end{proof}

\begin{prop}  \label{p-rosy-thorn-cbased}
Suppose $T$ is a real rosy first order theory, and $\thind$ has countably local character.   Then $\thind$ is countably based.
\end{prop}

\begin{proof}  By Result \ref{f-weakest} (1), Remark \ref{r-th-two-sided}, and Lemma \ref{l-cbased}, it is enough to
show that for all countable $A,B$ and small $C$ we have
\begin{equation}  \label{e-thorn}
A\thind_CB \Leftrightarrow (\forall^c C'\subseteq C)(\exists^c D\in [C',C])A\thind_DB.
\end{equation}
Since $\thind$ satisfies monotonicity, base monotonicity, transitivity, symmetry and countably local character, Proposition \ref{p-countably-local-implies} gives
the forward direction of (\ref{e-thorn}).

Suppose $A\nthind_CB$.  We follow the notation in the first paragraph of the proof of Lemma \ref{l-thorn-union}.  Suppose that $D\in[C_1,C]$ is countable.
Then, for each $i<m$, we have $\vec d_i\notin \acl(D\vec e_i)$ and $\{\psi_i(\vec x,\vec d_i')  :  \vec d_i'\equiv_{D\vec e_i}\vec d_i\}$ is $k_i$-inconsistent.  Thus each
$\psi_i(\vec x,\vec d_i)$ $\th$-divides over $D$.  It follows that $A\nthind_DB$, so the right hand side of (\ref{e-thorn}) fails.
\end{proof}

\section{Randomizations}

\subsection{The Theory $T^R$}

The \emph{randomization signature} $L^R$ is the two-sorted continuous signature
with sorts $\BK$ (for random elements) and $\BB$ (for events), an $n$-ary
function symbol $\l\varphi(\cdot)\rr$ of sort $\BK^n\to\BB$
for each first order formula $\varphi$ of $L$ with $n$ free variables,
a $[0,1]$-valued unary predicate symbol $\mu$ of sort $\BB$ for probability, and
the Boolean operations $\top,\bot,\sqcap, \sqcup,\neg$ of sort $\BB$.  The signature
$L^R$ also has distance predicates $d_\BB$ of sort $\BB$ and $d_\BK$ of sort $\BK$.
In $L^R$, we use ${\sa B},{\sa C},\ldots$ for variables or parameters of sort $\BB$. ${\sa B}\doteq{\sa C}$
means $d_\BB({\sa B},{\sa C})=0$, and ${\sa B}\sqsubseteq{\sa C}$ means ${\sa B}\doteq{\sa B}\sqcap{\sa C}$.

A pre-structure for $T^R$ will be a pair $\cu P=(\cu K,\cu E)$ where $\cu K$ is the part of sort $\BK$ and
$\cu E$ is the part of sort $\BB$.\footnote{In [BK], the set of events was denoted by $\cu B$, but we use $\cu E$ here
to reserve the letters $\cu A, \cu B, \cu C$ for subsets of $\cu E$.}
The \emph{reduction} of $\cu P$ is the pre-structure $\cu N=(\hat{\cu K},\hat{\cu E})$ obtained from
$\cu P$ by identifying elements at distance zero in the metrics $d_\BK$ and $d_\BB$,
 and the associated mapping from $\cu P$ onto $\cu N$ is called the \emph{reduction map}.
The \emph{completion} of $\cu P$ is the structure obtained by completing the metrics in the reduction of $\cu P$.
By a \emph{pre-complete-structure} we mean a pre-structure $\cu P$ such that the reduction of $\cu P$ is equal to the completion of $\cu P$.
By a \emph{pre-complete-model} of $T^R$ we mean a pre-complete-structure that is a pre-model of $T^R$.

In [BK], the randomization theory $T^R$ is defined by listing a set of axioms.
We will not repeat these axioms here, because it is simpler to give the following model-theoretic
characterization of $T^R$.

\begin{df}  \label{d-neat}
Given a model $\cu M$ of $T$, a \emph{neat randomization of} $\cu M$ is a pre-complete-structure $\cu P=(\cu L,\cu F)$ for $L^R$
equipped with an atomless probability space $(\Omega,\cu F,\mu )$ such that:
\begin{enumerate}
\item $\cu F$ is a $\sigma$-algebra with $\top,\bot,\sqcap, \sqcup,\neg$ interpreted by $\Omega,\emptyset,\cap,\cup,\setminus$.
\item $\cu L$ is a set of functions $a\colon\Omega\to M$.
\item For each formula $\psi(\vec{x})$ of $L$ and tuple
$\vec{a}$ in $\cu L$, we have
$$\l\psi(\vec{a})\rr=\{\omega\in\Omega:\cu M\models\psi(\vec{a}(\omega))\}\in\cu F.$$
\item $\cu F$ is equal to the set of all events
$ \l\psi(\vec{a})\rr$
where $\psi(\vec{v})$ is a formula of $L$ and $\vec{a}$ is a tuple in $\cu L$.
\item For each formula $\theta(u, \vec{v})$
of $L$ and tuple $\vec{b}$ in $\cu L$, there exists $a\in\cu L$ such that
$$ \l \theta(a,\vec{b})\rr=\l(\exists u\,\theta)(\vec{b})\rr.$$
\item On $\cu L$, the distance predicate $d_\BK$ defines the pseudo-metric
$$d_\BK(a,b)= \mu \l a\neq b\rr .$$
\item On $\cu F$, the distance predicate $d_\BB$ defines the pseudo-metric
$$d_\BB({\sa B},{\sa C})=\mu ( {\sa B}\triangle {\sa C}).$$
\end{enumerate}
\end{df}

Note that if $\cu H\prec\cu M$, then every neat randomization of $\cu H$ is also a neat randomization of $\cu M$.

\begin{df}  For each first order theory $T$, the \emph{randomization theory} $T^R$ is
the set of sentences that are true in all neat randomizations of models of $T$.
\end{df}

It follows that for each first order sentence $\varphi$, if $T\models\varphi$
then we have $T^R\models \l\varphi\rr\doteq \top$.


\begin{result}  \label{f-perfectwitnesses}  (Fullness, Proposition 2.7 in [BK]).
Every pre-complete-model $\cu N=(\cu K,\cu E)$ of $T^R$ has perfect witnesses, i.e.,
\begin{enumerate}
\item  For each first order formula $\theta(u,\vec v)$ and each $\vec{b }$ in $\cu K^n$ there exists $a \in\cu K$ such that
$$ \l\theta(a,\vec b)\rr \doteq
\l(\exists u\,\theta)(\vec{b })\rr;$$
\item For each ${\sa E}\in\cu E$ there exist $a ,b \in\cu K$ such that
${\sa E}\doteq\l a=b \rr$.
\end{enumerate}
\end{result}

The following results are proved in [Ke1], and are stated in the continuous setting in [BK].

\begin{result}  \label{f-complete} (Theorem 3.10 in [Ke1], and Theorem 2.1 in [BK]).
For every complete first order theory $T$, the randomization theory $T^R$ is complete.
\end{result}

\begin{result} \label{f-qe} (Strong quantifier elimination, Theorems 3.6 and 5.1 in [Ke1], and Theorem 2.9 in [BK])
Every formula $\Phi$ in the continuous language $L^R$
is $T^R$-equivalent to a formula with the same free variables
and no quantifiers of sort $\BK$ or $\BB$.
\end{result}

\begin{result}  \label{f-T^R} (Proposition 4.3 and Example 4.11 in [Ke1], and Proposition 2.2 and Example 3.4 (ii) in [BK]).
Every model $\cu M$ of $T$ has neat randomizations.
\end{result}

\begin{result}  \label{f-glue}  (Lemma 2.1.8 in [AGK])

Let $\cu P=(\cu K,\cu E)$ be a pre-complete-model of $T^R$ and let $a ,b \in\cu K$ and ${\sa B}\in\cu E$.
Then there is an element $c \in\cu K$ that agrees with $a $ on ${\sa B}$ and agrees with $b $ on $\neg{\sa B}$,
that is, ${\sa B}\sqsubseteq\l c =a \rr$ and $(\neg{\sa B})\sqsubseteq\l c =b \rr$.
\end{result}

\begin{df}  In Result \ref{f-glue}, we call $c$ a \emph{characteristic function of $\sa B$ with respect to $a,b$}.
\end{df}

Note that the distance between any two characteristic functions of an event $\sa B$ with respect to elements $a,b$ is zero.
In particular, in a model of $T^R$, the characteristic function is unique.

\begin{result}  \label{f-representation}  (Proposition 2.1.10 in [AGK])
Every model of $T^R$ is isomorphic to the reduction of a neat randomization $\cu P$ of a model of $T$.
\end{result}

\begin{lemma}  \label{l-big-embedding}
Let  $\cu N=(\cu K,\cu E)$ be a big model of $T^R$ and let $\cu M$ be a  model of $T$ of cardinality $\le\upsilon$.  There is a mapping $a\mapsto \ti a$ from $M$ into $\cu K$
with the following property:
\medskip

For each tuple $a_0,a_1,\ldots$ in $M$ and first order formula $\varphi(v_0.v_1,\ldots)$, if $\cu M\models\varphi(a_0,a_1,\ldots)$ then $\mu(\l \varphi(\ti a_0,\ti a_1,\ldots)\rr) =1$.
\end{lemma}

\begin{proof} This is proved by a routine transfinite induction using Fullness and saturation.
\end{proof}

\begin{result}  \label{f-TR-stable}  ([BK], Theorem 5.14)
The theory $T^R$ is stable if and only if $T$ is stable.
\end{result}

However, in [Be2], it is shown that the randomization of a simple, unstable theory is not simple.  We will see in Corollary \ref{c-rosy-downward}
that if $T^R$ is real rosy then $T$ is also real rosy.  We thus pose the na\"ive

\begin{question}
If $T$ is a real rosy first order theory, is the randomization $T^R$ also real rosy, or at least ``almost real rosy'' in some reasonable sense?
\end{question}

We do not even know the answer to the following question, where $\DLO$ is the theory of dense linear order without endpoints.

\begin{question}  \label{q-DLO-rosy} Is $\DLO^R$ real rosy?
\end{question}

\subsection{Blanket Assumptions}


From now on we will work within the big model $\cu N=( {\cu K}, {\cu E})$ of $T^R$.
We let $\cu M$ be the big model of $T$ and let $\cu P=(\cu L,\cu F)$ be a neat randomization of $\cu M$
with probability space $(\Omega,\cu E,\mu)$, such that $\cu N$ is the reduction of $\cu P$.
We may further assume that the probability space $(\Omega,\cu F, \mu)$ of $\cu P$ is complete (that is, every set
that contains a set of $\mu$-measure one belongs to $\cu F$), and that every
function $a\colon \Omega\to M$ that agrees with some $b\in \cu L$ except on a $\mu$-null subset of $\Omega$ belongs to $\cu L$.
The existence of $\cu P$ is guaranteed by Result \ref{f-representation} (Proposition 2.1.10 in [AGK]), and the further assumption
follows from the proof in [AGK].

We choose once and for all a mapping $a\mapsto \ti a$ from $M$ into $\cu K$ with the property stated in Lemma \ref{l-big-embedding},
and for each $A\subseteq M$ let $\ti A$ be the image of $A$ under this mapping.  Note that $\ti M$ has the discrete topology in $\cu K$, and hence is closed in $\cu K$.
For convenience, we also choose once and for all a pair of distinct elements $0,1\in M$
(but we do not assume that $L$ has constant symbols for $0,1$).  Thus $\mu(\l {\ti 0}\ne{\ti 1}\rr) = 1$.
For an event $\sa E\in\cu E$, we let $1_{\sa E}$ be the characteristic function of $\sa E$ with respect to ${\ti 0},{\ti 1}$; note that $1_{\sa E}\in\cu K$.

By saturation, $\cu K$ and $\cu E$ are large.  Hereafter, $A, B, C$ will always denote small subsets of $\cu K$.
For each element $\bo a\in {\cu K}$, we will also choose once and for all an element $a\in\cu L$ such that
the image of $a$ under the reduction map is $\bo a$.  It follows that for each first order formula $\varphi(\vec v)$,
$\l\varphi(\vec{\bo a})\rr$ in $\cu N$ is the image of $\l\varphi(\vec a)\rr$ in $\cu P$ under the reduction map.

For any small $A\subseteq\cu K$ and each $\omega\in\Omega$, we define
$$ A(\omega)=\{a(\omega)\colon \bo a\in  A\},$$
and let $\cl( A)$ denote the closure of $ A$ in the metric $d_\BK$.  When $\cu A\subseteq {\cu E}$,
$\cl(\cu A)$ denotes the closure of $\cu A$ in the metric $d_\BB$, and
$\sigma(\cu A)$ denotes the smallest $\sigma$-subalgebra of $ {\cu E}$ containing $\cu A$.
Since the cardinality $\upsilon$ of $\cu N$ is inaccessible, whenever $A\subseteq\cu K$ is small, the closure $\cl(A)$ and
the set of $n$-types over a $A$ is small.  Also, whenever $\cu A\subseteq\cu E$ is small, the closure $\cl(\cu A)$ is small.

If you wish to avoid the assumption that uncountable inaccessible cardinals exist, then instead of assuming that $\upsilon$ is inaccessible, you can assume only that $\upsilon$
is a strong limit cardinal with $\upsilon=\upsilon^{\aleph_0}$,
that $\cu N$ is an $\upsilon$-universal domain, and that $\cu M$ is an $\upsilon^+$-universal model of $T$ of cardinality $\upsilon$, which exists by Theorem 5.1.16 in [CK].

\subsection{Definability in $T^R$}

As explained in [AGK], in models of $T^R$ we need only consider definability over sets of parameters of sort $\BK$.

We write $\dcl_\BB( A)$ for the set of elements of sort $\BB$ that are definable over $ A$ in $\cu N$,
and write $\dcl( A)$ for the set of elements of sort $\BK$ that are definable over $ A$ in $\cu N$.
Similarly for $\acl_\BB( A)$ and $\acl( A)$.  We often use the following result without explicit mention.

\begin{result}  \label{f-acl=dcl}  ([AGK], Proposition 3.3.7, see also [Be2])
$\acl_\BB( A)=\dcl_\BB( A)$ and $\acl( A)=\dcl( A)$.
\end{result}

\begin{df}  We say that an event $\sa E$ is \emph{first order definable over $A$}, in symbols $\sa E\in\fo_\BB(A)$, if
$\sa E=\l\theta(\vec{\bo a})\rr$ for some formula $\theta$ of $L$ and some tuple $\vec{\bo a}\in A^{<\BN}$.
\end{df}

\begin{df}  A first order formula $\varphi(u,\vec v)$ is \emph{functional} if
$$T\models(\forall \vec v)(\exists ^{\le  1} u)\varphi(u,\vec v).$$

We say that $\bo b$ is \emph{first order definable over $ A$}, in symbols $\bo b\in\fo( A)$, if there is a functional formula
$\varphi(u,\vec v)$ and a tuple $\vec{\bo a}\in {A}^{<\BN}$ such that
$\l \varphi(b,\vec{a})\rr=\top$.
\end{df}

\begin{result} \label{f-separable}  ([AGK], Theorems 3.1.2 and  3.3.6)
$$\dcl_\BB( A)=\cl(\fo_\BB( A))=\sigma(\fo_\BB(A))\subseteq\cu E,\qquad \dcl( A)=\cl(\fo( A))\subseteq\cu K.$$
If $A$ is empty, then $\dcl_\BB(A)=\{\top,\bot\}$ and $\dcl(A)=\fo(A)$.
\end{result}

It follows that whenever $A$ is small, $\dcl(A)$ and $\dcl_\BB(A)$ are small.

Using the distance function between an element and a set, \ref{f-separable} can be re-stated as follows:

\begin{rmk}  \label{r-distance}
$$\sa E\in\dcl_\BB(A) \mbox{ if and only if } d_\BB(\sa E,\fo_\BB(A))=0,$$
and
$$\bo b\in\dcl(A) \mbox{ if and only if } d_\BK(\bo b,\fo(A))=0.$$
\end{rmk}

\begin{rmk}  \label{r-dcl-B}  For each small $A$,
$$\fo_\BB(\fo(A))=\fo_\BB(A),\quad\dcl_\BB(\dcl(A))=\dcl_\BB(A).$$
\end{rmk}

\begin{proof}
The first equation is clear, and the second equation follows from the first equation and Remark \ref{r-distance}.
\end{proof}

\begin{rmk}  \label{r-new}  If $B=\dcl^{\cu M}(A)$ then $\ti B=\fo(\ti A)=\dcl(\ti A)$.
\end{rmk}

\begin{proof}  For any first order functional formula $\psi(X,y)$, we have
$\cu M\models \psi(A,b)$ iff $\cu N\models\l\psi(\ti A,\ti b)\rr=\top$, and $\cu N\models\l(\exists ^{\le 1} y)\psi(\ti A,y)\rr=\top$.
Therefore $\ti B=\fo(\ti A)$.  For any two distinct elements $\ti b, \ti c$ of $\ti B$, we have $\cu M\models b\ne c$, so $\cu N\models\mu(\l\ti b\ne\ti c\rr)=1$
and hence $d_\BK(\ti b,\ti c)=1$.  Thus any two elements of $\ti B$ have distance $1$ from each other,
so $\ti B$ is closed in $\cu N$.  By Result \ref{f-separable}, $\fo(\ti A)=\cl(\fo(\ti A))=\dcl(\ti A)$.
\end{proof}

We will sometimes use the $\l \ldots\rr$ notation in a general setting.  Given a property $P(\omega)$, we write
$$\l P\rr=\{\omega\in\Omega\,:\,P(\omega)\},$$
and we say that
$$P(\omega) \mbox{ holds }\as.$$
if $\l P\rr$ contains a set  $\sa A\in\cu F$ such that $\mu(\sa A)=1$.
For example, $\l b\in\dcl^{\cu M}(A)\rr$ is the set of all $\omega\in\Omega$ such that $b(\omega)\in\dcl^{\cu M}(A(\omega))$.

\begin{result}  \label{f-pointwisemeasurable}  ([AGK], Lemma 3.2.5)
If $A$ is countable, then
$$\l b\in\dcl^{\cu M}(A)\rr= \bigcup\{\l\theta(b,\vec a)\rr\colon\theta(u,\vec v) \mbox{ functional, } \vec{\bo a}\in A^{<\BN}\},$$
and  $\l b\in\dcl^{\cu M}(A)\rr\in\cu F$.
\end{result}

It follows that for each countable $A$, $b(\omega)\in\dcl(A(\omega))\as.$ if and only if $\mu(\l b\in\dcl(A)\rr)=1$.

\begin{df} \label{d-pointwise-def}
 We say that $\bo b$ is \emph{pointwise definable over $A$}, in symbols $\bo b\in\dcl^\omega(A)$, if
$$\mu(\l b\in\dcl^{\cu M}(A_0)\rr)=1$$
for some countable $A_0\subseteq A$.

We say that $\bo b$ is \emph{pointwise algebraic over $A$}, in symbols $\bo b\in\acl^\omega(A)$, if
$$\mu(\l b\in\acl^{\cu M}(A_0)\rr)=1$$
for some countable $A_0\subseteq A$.
\end{df}

\begin{rmk}  \label{r-dcl-omega}
$\dcl^\omega$ and $\acl^\omega$ have countable character, that is,
$\bo b\in\dcl^\omega(A)$ if and only if $\bo b\in\dcl^\omega(A_0)$ for some countable $A_0\subseteq A$, and similarly for $\acl^\omega$.
\end{rmk}

\begin{result}  \label{f-dcl3}  ([AGK], Corollary 3.3.5) For any element $\bo b\in {\cu K}$,
$\bo{b}$ is definable over $ A$ if and only if:
\begin{enumerate}
\item $\bo b$ is pointwise definable over $A$;
\item $\fo_\BB(\bo b A)\subseteq\dcl_\BB(A)$.
\end{enumerate}
\end{result}

\begin{cor}  \label{c-pointwise-alg-def}
In $\cu N$ we always have
$$\acl(A)=\dcl(A)\subseteq\dcl^\omega(A)=\dcl^\omega(\dcl^\omega(A))\subseteq\acl^\omega(A)=\acl^\omega(\acl^\omega(A)).$$
\end{cor}

The following proposition gives a warning: the set $\dcl^\omega(A)$ is almost always large.
(By contrast, it is well-known that for any small set $A$, $\acl(A)$ and $\dcl(A)$ are small--this follows from Fact \ref{f-algebraic}.)

\begin{prop}  \label{p-large}  If $|A|>1$, or even if $|\dcl^\omega(A)|>1$, then $\dcl^\omega(A)$ is large.
\end{prop}

\begin{proof}  We may assume that $A$ is countable.  Take two elements $\bo a\ne \bo b\in\dcl^\omega(A)$.  Then $\mu(\l a\ne b\rr)=r>0$.
By Result \ref{f-glue}, for each event $\sa E\in\cu E$ the characteristic function $\chi_{\sa E}$ of $\sa E$ with respect to $\bo a,\bo b$
belongs to $\cu K$, and
$\mu(\l \chi_{\sa E}\in \dcl(A)\rr)=1$, so $\chi_{\sa E}\in\dcl^\omega(A)$. For each $n$ there is a set of $n$ events $\sa E_1,\ldots,\sa E_n$
such that $d_\BK(\chi_{\sa E_i},\chi_{\sa E_j})=d_\BB(\sa E_i,\sa E_j)=r/2$ whenever $i<j\le n$.  Then by saturation, the set
$\dcl^\omega(A)$ has cardinality $\ge\upsilon$ and hence is large.
\end{proof}

\begin{cor}  \label{c-subset}  Let $A$ be a countable subset of $\cu K$ with $|A|>1$.  Then the set of all small $B$ such that
$B(\omega)\subseteq A(\omega) \as.$ is large and contains every subset of $A$.
\end{cor}

\begin{proof}  Similar to the proof of Proposition \ref{p-large}.
\end{proof}

\section{Independence in $T^R$}

In this section we study the relations $\ind[a]$, $\ind[M]$, $\thind$, and $\ind[d]$ over $\cu N$.
In the two-sorted metric structure $\cu N$, algebraic independence is defined by
$$ A\ind[a]_C B\Leftrightarrow [\acl(AC)\cap\acl(BC)=\acl(C) \wedge \acl_\BB(AC)\cap\acl_\BB(BC)=\acl_\BB(C)].$$
The other special ternary relations $\ind[M] \ $ and $\thind$ are defined in terms of $\ind[a]$ over $\cu N$,
as in Definition \ref{d-special-ind}.

\subsection{Downward Results}

\begin{df}  We say that $T$ has $\acl=\dcl$ if for every set $A$ in $\cu M$ we have  $\acl^{\cu M}(A)=\dcl^{\cu M}(A)$.
\end{df}

For example, every theory with a definable linear ordering has $\acl=\dcl$, but the theory of algebraically closed fields does not have $\acl=\dcl$.

We show that if $T$ has $\acl=\dcl$, then each  special independence relation we have considered holds for subsets of $\cu M$ only if the corresponding
relation holds for their images of the sets under the mapping $a\mapsto\ti a$.

\begin{prop}  \label{p-downward}  Suppose $T$ has $\acl=\dcl$,
 $\ind[I]$ is one of the relations $\ind[a],\, \ind[M], \,\thind$, and $A, B, C$ are small subsets of $M$.
 If $\ti A\ind[I]_{\ti C} \ti B$ holds in $\cu N$, then $A\ind[I]_C B$ holds in $\cu M$.
\end{prop}

\begin{proof}  Suppose first that $\ti A\ind[a]_{\ti C} \ti B$ in $\cu N$. Then
$$\acl(\ti A \ti C)\cap\acl(\ti B\ti C)=\acl(\ti C).$$
Let
$$A'=\acl^{\cu M}(AC), B'=\acl^{\cu M}(BC), C'=\acl^{\cu M}(C).$$
By Result \ref{f-acl=dcl} and Remark \ref{r-new},
 $A'=\dcl^{\cu M}(AC)$, and $\ti{A'}=\dcl(\ti A\ti C)=\acl(\ti A\ti C)$.  Similarly for $B'$ and $C'$.
Therefore $\ti {A'}\cap \ti {B'}= \ti {C'}$.  It follows that $A'\cap B'=C'$, which means that $A\ind[a]_C B$ in $\cu M$.

Now suppose that $\ti A\ind[M]_{\ti C} \ti B$ in $\cu N$.  Let $D\in[C,\acl^{\cu M}(BC)]$ in $\cu M$.  Then $D\subseteq\dcl^{\cu M}(BC)$,
so by Remark \ref{r-new}, $\ti D\subseteq\dcl(\ti B\ti C)=\acl(\ti B\ti C)$.  Hence $\ti A\ind[a]_{\ti D} \ti B$ in $\cu N$, so
$A\ind[a]_D B$  and $A\ind[M]_C B$ in $\cu M$.

Suppose that $\ti A\thind_{\ti C}\ti B$ in $\cu N$.  Let $B\subseteq D$ with $D$ small in $\cu M$.  Then $\ti B\subseteq \ti D$ in $\cu N$ and $\ti D$ is small.
Hence there exists $A_1\equiv_{\ti B\ti C} \ti A$ such that $A_1\ind[M]_{\ti C}\ti D$ in $\cu N$.  Then for every $E_1\in[\ti C,\acl(\ti C\ti D)]$
we have $A_1\ind[a]_{E_1}\ti D$ in $\cu N$.  By Remark \ref{r-new} and the assumption that $T$ has $\acl=\dcl$, $E_1\in[\ti C,\acl(\ti C\ti D)]$ if and only if
 $E_1=\ti E$ for some $E\in[C,\acl(CD)]$.  So $A_1\ind[a]_{\ti E}\ti D$ in $\cu N$ for every $E\in[C,\acl(CD)]$.
Let $F=\acl(CD)=\dcl(CD)$ in $\cu M$, so  by Remark \ref{r-new}, $\ti F=\acl(\ti C\ti D)$ in $\cu N$.
By saturation in $\cu M$, there is a set $G\subseteq M$ such that for every first order formula $\varphi(X,\ti F)$ such that $\l\varphi(A_1,\ti F)\rr=\top$
in $\cu N$, we have $\cu M\models \varphi(G,F)$.  Then in $\cu N$ we have $\ti G\equiv_{\ti B\ti C} A_1\equiv_{\ti B\ti C} \ti A$,
and $\ti G\ind[a]_{\ti E}\ti D$  for every $E\in[C,\acl(CD)]$. Therefore in $\cu M$ we have $G\equiv_{BC} A$ and $G\ind[a]_E D$  for every $E\in[C,\acl(CD)]$.
It follows that $G\ind[M]_C D$ and $A\thind_C B$ in $\cu M$.
\end{proof}

The following Corollary can be compared with Corollary 7.9 of [EG], which says that if $T^R$ is maximally real rosy then $T$ is real rosy.

\begin{cor}  \label{c-rosy-downward}
Suppose $T$ has $\acl=\dcl$ and $T^R$ is real rosy.  Then $T$ is real rosy.
\end{cor}

\begin{proof}  By hypothesis, the relation $\thind$ over $\cu N$ has local character, with some bound $\kappa^{\cu N}(A)$.
We show that $\thind$ over $\cu M$ has local character with bound $\kappa^{\cu M}(A)=\kappa^{\cu N}(\ti A)$.
Let $A, B\subseteq M$ be small.  Then there is a set $C'\subseteq \ti B$ such that $|C'|<\kappa^{\cu N}(\ti A)$ and $\ti A\thind_{C'} \ti B$ in $\cu N$.
It is  clear that $C'=\ti C$ for some set $C\subseteq B$.  By Proposition \ref{p-downward} we have $A\thind_C B$ in $\cu M$.
\end{proof}

\begin{ex}
Let $T$ be the theory of an equivalence relation $E$ such that there are infinitely many equivalence classes and each equivalence class has cardinality $3$.
Then $T$ is stable and even categorical in every infinite cardinality, but does not have $\acl=\dcl$.  Let $a,b$ be elements of $\cu M$ such that
$E(a,b)$ but $a\ne b$.  Then $\ti a\ind[M]_\emptyset \ti b$ in $\cu N$, but $a\nind[a]_\emptyset b$ in $\cu M$.
Thus Proposition \ref{p-downward} fails for $\ind[a]$ and $\ind[M] \ $ when the hypothesis that $T$ has $\acl=\dcl$ is removed.
\end{ex}

\begin{question} Do Proposition \ref{p-downward} for $\thind$ and Corollary \ref{c-rosy-downward} hold without the  hypothesis that $T$ has $\acl=\dcl$?
\end{question}

We now prove the analogous results for the relation $\ind[d]$, even without the hypothesis that $T$ has $\acl=\dcl$.

\begin{prop}  \label{p-d-downward}  Suppose $A,B,C$ are small subsets of $M$, and $\ti A\ind[d]_{\ti C} \ti B$  holds in $\cu N$.
Then $A\ind[d]_C B$ holds in $\cu M$.
\end{prop}

\begin{proof} Assume that $A\ind[d]_C B$ fails in $\cu M$. We may assume that $A$ and $B$ are finite. There is a
first order formula $\varphi(X,Y,Z)$ such that $\cu M\models \varphi(A,B,C)$ and $\varphi(X,B,C)$
divides over $C$.  Then there exists $k\in\BN$ and a $C$-indiscernible sequence $\< B_i\>_{i\in\BN}$ such that $B_0=B$ and
$\{\varphi(X,B_i,C)\colon i\in\BN\}$ is $k$-contradictory.
Since the mapping $d\mapsto\ti d$ has the property stated in Lemma \ref{l-big-embedding}, it follows that for $\mu$-almost all $\omega\in\Omega$,
$\cu M\models \varphi(\ti A(\omega),\ti B(\omega),\ti C(\omega))$, $\< \ti B_i(\omega)\>_{i\in\BN}$ is $\ti C$-indiscernible in $\cu M$,
$\ti B_0(\omega)=\ti B(\omega)$, and $\{\varphi(X,B_i(\omega),C(\omega))\colon i\in\BN\}$ is $k$-contradictory in $\cu M$.
Therefore the continuous formula $1-\mu(\l\varphi(X,\ti B,\ti C)\rr)$ divides over $\ti C$ in $\cu N$,
so $\ti A\ind[d]_{\ti C} \ti B$ fails in $\cu N$.
\end{proof}

\begin{cor}  \label{c-simple-downward}
If the relation $\ind[d]$ over $\cu N$ has local character, then the relation $\ind[d]$ over $\cu M$ has local character.
\end{cor}

In this connection, we recall that by Result \ref{f-TR-stable}, $T$ is stable if and  only if $T^R$ is stable, and
Ben Yaacov [Be2] showed  that if $T^R$ is simple then $T^R$ is stable, so $T$ is also stable.

\subsection{Independence in the Event Sort}  \label{s-event}

The one-sorted theory $\APr$ of atomless probability algebras is studied in the papers [Be1], [Be2], and [BBHU].
By Fact 2.10 in [Be2], for every model $\cu N=(\cu K,\cu E)$ of $T^R$, the event sort $(\cu E,\mu)$ of $\cu N$
 is a model of $\APr$.  For each cardinal $\kappa$, if $\cu N$ is $\kappa$-saturated then $( {\cu E},\mu)$ is $\kappa$-saturated.
For every set $\cu A\subseteq\cu E$, we have
$\acl(\cu A)=\dcl(\cu A)=\sigma(\cu A)$ in $(\cu E,\mu)$.  The algebraic independence relation in $(\cu E,\mu)$ is the relation
$$ \cu A\ind[a]_{\cu C} \cu B \Leftrightarrow \sigma(\cu A\cu C)\cap\sigma(\cu B\cu C)=\sigma(\cu C).$$

\begin{result}  \label{f-probability-algebra}  (Ben Yaacov [Be1])
The theory $\APr$ is separably categorical, admits quantifier elimination, and is stable.  Its unique strict independence relation $\ind[f]$
is the relation of probabilistic independence, given by $\cu A\ind[f]_{\cu C} \cu B$ if and only if
$$\mu[\sa A\sqcap\sa B | \sigma(\cu C)]=\mu[\sa A | \sigma(\cu C)]\mu[\sa B | \sigma(\cu C)] \as \mbox{ for all } \sa A\in \sigma(\cu A), \sa B\in\sigma(\cu B).$$
\end{result}

\begin{df}  \label{d-event-indep}
For small $A, B, C\subseteq \cu K$, define
 $$A\ind[a\BB]_C \ \ B\Leftrightarrow\acl_\BB(AC)\cap\acl_\BB(BC)=\acl_\BB(C).$$
\end{df}

Given sets $C,D\subseteq \cu K$, it will be convenient to introduce the notation  $\cu D=\fo_\BB(D)$ and $\cu D_C=\fo_\BB(DC)$.

\begin{rmk}  \label{r-ind-B}
\noindent\begin{enumerate}
\item By Result \ref{f-separable}, $\acl_\BB(D)=\sigma(\cu D)=\acl(\cu D)$.
\item $A\ind[a\BB]_C \ \ B$ in $\cu N\Leftrightarrow \cu A_C\ind[a]_{\cu C} \cu B_C$ in $(\cu E,\mu)$.
\end{enumerate}
\end{rmk}

\begin{prop} \label{p-event-aB}
The relation $\ind[a\BB] \ \ $ over $\cu N$ satisfies all the axioms for a countable independence relation except base monotonicity.
It also has symmetry, the countable union property, and countably local character.
\end{prop}

\begin{proof}  Symmetry, invariance, monotonicity, normality, countable character, and the countable union property are clear.

Transitivity: Suppose $C\in[D,B]$, $B\ind[a\BB]_C \ A$, and $C\ind[a\BB]_D \ A$.
Then $\cu A_C$,  $\cu B_C$, and $\cu C$, are small, $\cu C\in[\cu D,\cu B]$,
and $\cu A_D\subseteq\cu A_C$.  We have
$\cu B_C\ind[a]_{\cu C} \cu A_C$ and $\cu C_D\ind[a]_{\cu D} \cu A_D$.  By monotonicity of $\ind[a]$, $\cu B_D\ind[a]_{\cu C} \cu A_D$.  Then
by transitivity of $\ind[a]$,  $\cu B_D\ind[a]_{\cu D} \cu A_D$, so $B\ind[a\BB]_D \ A$.

By Remarks \ref{r-fe-vs-ext}, to prove extension for $\ind[a\BB] \ $ it suffices to prove full existence for $\ind[a\BB] \ \ $.

Full existence: For any $A,B,C$, we must show that there exists $A'\equiv_C A$ such that $A'\ind[a\BB]_C B$.
We may assume that $C\subseteq A$, so that $\cu A=\cu A_C$.  Since $\ind[a]$ has full existence in $(\cu E,\mu)$,
there exists $\cu A'\subseteq\cu E$ such that $\cu A'\equiv_{\cu C} \cu A$ and $\cu A'\ind[a]_{\cu C} \cu B_C$ in $(\cu E,\mu)$.
Note that every quantifier-free formula of $T^R$ with parameters in $\cu A\cup C$ has the form $f(\mu(\tau_1),\ldots,\mu(\tau_m))$
where $f\,:\,[0,1]^m\to[0,1]$ is continuous, and each $\tau_i$ is a Boolean term of the form
$$\tau_i(\l\theta_1(C)\rr,\ldots,\l\theta_n(C)\rr,\sa A_1,\ldots,\sa A_k)$$
with $\sa A_1,\ldots\sa A_k\in\cu A$.
By quantifier elimination, every formula of $T^R$ with parameters in $\cu A\cup C$ is equivalent to a formula of that form.
Therefore we have $\cu A'\equiv_C\cu A$ in $\cu N$.  Add a constant symbol to $\cu N$ for each $a\in A$, $\sa F\in\cu A$, and
$\sa F'\in\cu A'$. Add a set of variables $A'=\{\bo a'\colon \bo a\in A\}$.  Consider the set of conditions
$$ \Gamma=\{d_\BB(\sa F',\l\theta(\vec{\bo a}',\vec{\bo c})\rr)=0\colon \cu N\models d_\BB(\sa F,\l\theta(\vec{\bo a},\vec{\bo c})\rr)=0\}$$
with the set of variables $A'$.   It follows from fullness that every finite
subset of $\Gamma$ is satisfiable by some $A'$ in $\cu N$.  Then by  saturation, $\Gamma$ is satisfied by some $A'$ in $\cu N$.
By quantifier elimination we have $A\equiv_C A'$ in $\cu N$, and by definition, $A'\ind[a\BB]_C B$.

Local character:  Let $A, B$ be small subsets of $\cu K$.  We prove local character with bound $\lambda=(|A|+\aleph_0)^+$.
In the theory of $(\cu E,\mu)$, and even in the theory of $(\cu E,\mu)$ with fewer than $\lambda$ additional constant symbols,
$\ind[a]$ has local character with bound $\kappa(\cu A)=(|\cu A|+\aleph_0)^+$.
Note that $|\cu A|<\lambda$.  Let $C_0=\emptyset$.
We construct an increasing chain $C_0\subseteq C_1\subseteq\cdots$ of subsets of $B$ such that $|C_n|<\lambda$ as follows.
Given $C_n$, we note that $|\cu A_{C_n}|<\lambda$ and $|\cu C_n|<\lambda$, and $\cu C_n\subseteq\cu B$.  By local character for $\ind[a]$ in the
theory of $(\cu E,\mu)_{\cu C_n}$, there is a set $\cu D_{n+1}\subseteq \cu B$ such that $|\cu D_{n+1}|<\lambda$ and $\cu A_n\ind[a]_{\cu D_{n+1}} \cu B$.
Since $(\cu E,\mu)_{\cu C_n}$ has a constant symbol for each element of $\cu C_n$, we may assume that $\cu C_n\subseteq \cu D_{n+1}$.
Since $L$ is countable, there is a set $C_{n+1}\in[C_n,B]$ such that $|C_{n+1}|<\lambda$ and $\cu D_{n+1}\subseteq\cu C_{n+1}$.
Now let $C=\bigcup_n C_n$.  Then $\cu C=\bigcup_n\cu C_n$, $C\subseteq B$, and $|C|<\lambda$.
For each finite $\cu A'\subseteq \cu A_C$, we have $\cu A'\subseteq\cu A_{C_n}$ for some $n$, so
$$(\acl(\cu A')\cap \acl(\cu B))\subseteq \acl(\cu C_{n+1})\subseteq\acl(\cu C).$$
Therefore $\cu A_C\ind[a]_{\cu C} \cu B_C$, and hence $A\ind[a\BB]_C \ B$.
\end{proof}

\begin{prop}  \label{p-a-B-cbased}   The relation $\ind[a\BB] \ $ is countably based.
\end{prop}

\begin{proof}  The proof is the same as the proof of Proposition \ref{p-a-cbased} except that $\acl$ is replaced everywhere by $\acl_\BB$.
\end{proof}

The following two results show that $\ind[a\BB] \ \ $ never has base monotonicity, and is almost never anti-reflexive.

 \begin{prop}  \label{p-B-nobase-monotonicity}
For every $T$, $\ind[a\BB]$ \ \ does not have base monotonicity.
 \end{prop}

 \begin{proof}  Since $\mu$ is atomless, there are two independent events $\sa D, \sa F$ in $\cu E$ of probability $1/2$.
 Let $\sa E=\sa D\sqcap\sa F$. $\bo a=1_{\sa D}$, $\bo b = 1_{\sa E}$, and $\bo c=1_{\sa F}$.  Then
 $$\acl_\BB(\bo a)=\sigma(\{\sa D\}),$$
 $$\acl_\BB(\bo c)=\sigma(\{\sa F\}),$$
 $$\acl_\BB(\bo a\bo c)=\sigma(\{\sa D,\sa F\}),$$
 $$\acl_\BB(\bo b \bo c)=\sigma(\{\sa E,\sa F\}).$$
 It follows that $\bo a\ind[a\BB]_{\emptyset} \ \bo b\bo c$ but $\bo a\nind[a\BB]_{\bo c} \ \bo b\bo c$,
 so $\ind[a\BB]$ \ \ does not have  base monotonicity.
 \end{proof}

\begin{prop}  \label{p-a-B-not-anti-reflexive}
Suppose that $T$ has either an infinite model, or a finite model with an element that is not definable without parameters.
Then $\ind[a\BB]$ \ \ is not anti-reflexive.
\end{prop}

\begin{proof}
It follows from the hypotheses that $\cu M$ has an element $a$ whose type $p$ is not realizable by a definable element over $\emptyset$.
Then $\mu(\l\varphi({\ti a})\rr)=1$ for each $\varphi(v)\in p$, so
$\fo_\BB({\ti a})=\sigma(\emptyset)$.   Hence
${\ti a} \ind[a\BB]_{\emptyset}{\ti a}$.  But ${\ti a}\notin\dcl(\emptyset)=\acl(\emptyset)$ by Results \ref{f-dcl3} and \ref{f-acl=dcl},
so $\ind[a\BB]$ \ \ is not anti-reflexive.
 \end{proof}

We now consider the analogue of the forking independence relation $\ind[f]$ in the event sort when the theory $T$ is stable.

\begin{df} Suppose $T$ is stable.  For all $A, B, C\subseteq\cu K$, define
$$  A\ind[f\BB]_C \ B \Leftrightarrow \cu A_C\ind[f]_{\cu C} \cu B_C \mbox{ in } (\cu E, \mu).$$
\end{df}

\begin{lemma}  \label{l-fB-basic}
If $T$ is stable, then $\ind[f\BB] \ \ $ satisfies the basic axioms, symmetry, finite character, and the countable  union property.
Moreover, $\ind[f\BB] \ \Rightarrow \ind[a\BB] \ \ $.
\end{lemma}

\begin{proof}  By Result \ref{f-probability-algebra}, the theory $\APr$ of $(\cu E,\mu)$ is stable, so $\ind[f]$ is
an independence relation over $(\cu E,\mu)$.  It follows easily that $\ind[f\BB] \ \ $ satisfies invariance, monotonicity,
base monotonicity, normality, finite character, and symmetry.  Transitivity is proved as in the proof of Proposition \ref{p-event-aB}.
Since $\ind[f]\Rightarrow \ind[a]$ in $(\cu E,\mu)$, it follows at once that $\ind[f\BB] \ \Rightarrow\ind[a\BB] \ $ in $\cu N$.

Countable union property:
By Result \ref{f-probability-algebra}, $\APr$ is stable, so over $(\cu E,\mu)$, $\ind[f]=\ind[d]$.  By Proposition \ref{l-d-union},
$\ind[f]$ has the countable union property over $(\cu E,\mu)$.  Suppose $A,B, C$ are countable, $C=\bigcup_n C_n$, and $C_n\subseteq C_{n+1}$
and $A\ind[f\BB]_{C_n} B$ for each $n$.  By monotonicity for $\ind[f]$, whenever $n\le m$ we have $\cu A_{C_n}\ind[f]_{\cu C_m} \cu B_{C_n}$.
By the  countable union property for $\ind[f]$ over $(\cu E,\mu)$, for each $n$ we have $\cu A_{C_n}\ind[f]_{\cu C} \cu B_{C_n}$.
Then by finite character and monotonicity for $\ind[f]$, it follows that $\cu A_{C}\ind[f]_{\cu C} \cu B_{C}$, so $A\ind[f\BB]_C \ B$, so $\ind[f\BB] \ \ $ has
the countable union property.
\end{proof}

\begin{lemma}  \label{l-B-stable-implies}  If $T$ is stable, then in $\cu N$ we have $\ind[f]\Rightarrow \ind[f\BB] \ \ $.
\end{lemma}

\begin{proof}  Let $A, B, C$ be small subsets of $\cu K$, and suppose that $A\nind[f\BB]_C \ B$ in $\cu N$.
Then $\cu A_C\nind[f]_{\cu C} \cu B_C$ in $(\cu E,\mu)$.
By Result \ref{f-probability-algebra}, the theory $\APr$ of $(\cu E,\mu)$ is stable, so
there is a continuous formula $\Phi(\vec X,\cu B_C,\cu C)$ and a tuple $\vec A$ in $\cu A_C$ such that
$\Phi(\vec A,\cu B_C,\cu C)=0$ and $\Phi(\vec X,\cu B_C,\cu C)$ divides over $\cu C$ in $(\cu E,\mu)$.
Let $\Psi(\vec a,B,C)$ be a continuous formula in $L^R$ with parameters in $\cu K$ formed by replacing the elements of
$\vec A, \cu B_C, \cu C$ by equal elements of the form $\l\theta_1(A,C)\rr, \l\theta_2(B,C)\rr, \l\theta_3(C)\rr$ respectively,
and let $|\vec x|=|\vec a|$. It follows by saturation that $\Psi(\vec a,B,C)=0$ and $\Psi(\vec x,B,C)$ divides over $C$ in $\cu N$.
Therefore $A\nind[d]_C B$ in $\cu N$.  By Result \ref{f-TR-stable}, $T^R$ is stable, so $A\nind[f]_C B$ in $\cu N$.
\end{proof}

\begin{prop}  \label{p-B-stable-independence}
If $T$ is stable, then $\ind[f\BB] \ \ $ is an independence relation over $\cu N$, has countably local character, and is countably based.
\end{prop}

\begin{proof}
By Lemma \ref{l-B-stable-implies}, Lemma \ref{l-fB-basic},  Remark \ref{r-weaker}, $\ind[f\BB] \ \ $ is an independence relation over $\cu N$ with countably local character.
To prove that $\ind[f\BB] \ \ $ is countably based, we argue as in the proof of Proposition \ref{p-d-cbased}.  As in that proof,
it is enough to check that, for countable $A$, $B$ and small $C$, we have
$$
A\ind[f\BB]_C \ B \Leftrightarrow (\forall^c C'\subseteq C)(\exists^c D\in [C',C])A\ind[f\BB]_D \ B.
$$
Fix such $A$, $B$, $C$.  The forward direction follows from Corollary \ref{c-preserve}.  For the other direction, suppose that $A\nind[f\BB]_C \ B$.
Then $\cu A_C\nind[f]_{\cu C} \ \cu B_C$ over $(\cu E,\mu)$.  Hence for some tuple $\vec A$ in $\cu A_C$ and some continuous formula
$\Phi(\vec X, \cu B_C,\cu C)$, $(\cu E,\mu)\models \Phi(\vec A,\cu B_C,\cu C)=0$ and $\Phi(\vec X,\cu B_C,\cu C)$ divides over $\cu C$.
Take a countable (even finite) $C'\subseteq C$ such that $\Phi(\vec x,\cu B_C,\cu C)=\Phi(\vec x,\cu B_{C'},\cu C')$.
Then for any countable $D\in[C',C]$, $\Phi(\vec x,\cu B_C,\cu C)$ divides over $\cu D$.  Therefore $\cu A_D\nind[f]_{\cu D} \cu B_D$ over $(\cu E,\mu)$,
so $A\nind[f\BB]_C \ B$.
\end{proof}

\begin{cor}  \label{c-f-B-not-anti-reflexive}
Suppose that $T$ is stable, and has either an infinite model, or a finite model with an element that is not definable without parameters.
Then $\ind[f\BB]$ \ \ is not anti-reflexive.
\end{cor}

\begin{proof}  Same as the proof of Proposition \ref{p-a-B-not-anti-reflexive}.
\end{proof}

\subsection{Algebraic Independence in $T^R$}  \label{s-alg-t^R}

By the definition of algebraic independence in the two-sorted metric structure $\cu N$ and Result \ref{f-acl=dcl}, $ A\ind[a]_C B$ if and only if
$$[\dcl(AC)\cap\dcl(BC)=\dcl(C)] \wedge [\dcl_\BB(AC)\cap\dcl_\BB(BC)=\dcl_\BB(C)].$$

\begin{rmks}  \label{r-simplified}
If $\ti 0, \ti 1\in C$, then
$$A\ind[a]_C B\Leftrightarrow \dcl(AC)\cap\dcl(BC)=\dcl(C) .$$
\end{rmks}

\begin{proof}  Suppose  $\ti 0, \ti 1\in C$, $\dcl(AC)\cap\dcl(BC)=\dcl(C)$ and $\sa E\in\dcl_\BB(AC)\cap\dcl_\BB(BC)$.
Then $1_{\sa E}\in\dcl(AC)\cap\dcl(BC)$, so $1_{\sa E}\in\dcl(C),$ and hence
$\sa E\in\dcl_\BB(C)$.  Therefore $\dcl_\BB(AC)\cap\dcl_\BB(BC)=\dcl_\BB(C).$
\end{proof}

We have seen in Proposition \ref{p-alg-indep} that $\ind[a]$ over $\cu N$ satisfies symmetry and all axioms for a strict countable independence relation
except  perhaps for base monotonicity and extension.  The relation $\thind$ is always stronger than $\ind[a]$.  Whenever the randomization theory
$T^R$ is real rosy, $\thind$ has full existence by Remarks \ref{r-fe-vs-ext}, and hence $\ind[a]$ over $\cu N$ has full existence.
Here is another sufficient condition for $\ind[a]$ over $\cu N$ to have full existence.

\begin{thm} \label{t-fe} Suppose $T$ has $\acl=\dcl$.  Then the relation $\ind[a]$ over $\cu N$ has full existence and extension.
\end{thm}

\begin{proof}
By Remarks \ref{r-fe-vs-ext} and Proposition \ref{p-alg-indep}, if $\ind[a]$ over $\cu N$ has full existence, then it has extension.
To prove full existence,
we must show that for all small $A,B,C$, there is $A'\equiv_{C} A$ such that
$$[\dcl(A'C)\cap\dcl(BC)=\dcl(C)] \wedge [\dcl_\BB(A'C)\cap\dcl_\BB(BC)=\dcl_\BB(C)] .$$
In view of Fact \ref{f-definableclosure}  and Remark \ref{r-dcl-B}, we may assume without loss of generality that
$C=\acl(C)$, $A=\acl(AC)\setminus \acl(C)$, and $B=\acl(BC)\setminus \acl(C)$.
Then $C=\dcl(C)$, $A=\dcl(AC)\setminus \dcl(C)$, and $B=\dcl(BC)\setminus \dcl(C)$.    By Proposition \ref{p-event-aB}, the relation
$\ind[a\BB] \ \ $ over $\cu N$ has full existence.  Therefore we may also assume that $ A\ind[a\BB]_C \  B.$ By Result \ref{f-acl=dcl},
$$\dcl_\BB(AC)\cap\dcl_\BB(BC)=\dcl_\BB(C).$$
So it suffices to show that there is $A'\equiv_CA$ such that
$$A'\cap B=\emptyset \wedge \dcl_\BB(A'C)=\dcl_\BB(AC).$$

For each element $\bo a\in A$, we
define $\varepsilon(\bo a)$ as the infimum of all the values $1-\mu(\l a\in\dcl^{\cu M}(D)\rr)$ over all countable $D\subseteq C$.
Note that $\varepsilon(\bo a)=0$ if and only if $\bo a$ is pointwise definable over some countable subset of $C$.
Add a constant symbol for each $\bo a\in A, \bo b\in B$, and $\bo c\in C$.  For each $\bo a\in A$, add a variable $\bo a'$.
Consider the set $\Gamma$ of all conditions of the form
$$\l\theta(\vec{\bo a},\vec{\bo c})\rr=\l\theta(\vec{\bo a}',\vec{\bo c})\rr\wedge
\bigwedge_{i\le|\vec{\bo a}|}d_\BK({\bo a}'_i,\bo b)\ge \varepsilon({\bo a}_i)$$
where $\theta$ is an $L$-formula, $\vec{\bo a}\in A^{<\BN}, \vec {\bo c}\in C^{<\BN}$, and $\bo b\in B$.

\

\emph{Claim 2}.  For every finite subset $\Gamma_0$ of $\Gamma$, there is a set  $A'=\{{\bo a}'\colon \bo a\in A\}$ that satisfies $\Gamma_0$
in $\cu N_{ABC}$.

\

\emph{Proof of Claim}:  Let $A_0, B_0, C_0$ be the set of elements of $A, B, C$ respectively that occur in $\Gamma_0$.
Then $A_0, B_0, C_0$ are finite.  If $A_0$ is empty, then $\Gamma_0$ is trivially satisfiable in ${\cu N}_{ABC}$,
so we may assume that $A_0$ is non-empty.   Let
$$A_0=\{{\bo a}_0,\ldots,{\bo a}_n\},\vec{\bo a}=\<{\bo a}_0,\ldots,{\bo a}_n\>,
C_0=\{{\bo c}_0,\ldots,{\bo c}_k\},\vec{\bo c}=\<{\bo c}_0,\ldots,{\bo c}_k\>.$$
Let $\Theta_0$ be the set of all sentences that occur on the left side of an equation in $\Gamma_0$.  Then $\Theta_0$ is finite.  By combining tuples,
 we may assume that each sentence in $\Theta_0$ has the form $\theta(\vec{\bo a},\vec{\bo c})$.

Since the algebraic independence relation on$\cu M$ satisfies full existence, and $T$ has $\acl=\dcl$, for each $\omega\in\Omega$ there exists
$$G_0(\omega)=\{g_0(\omega),\ldots,g_n(\omega)\}\subseteq M$$
such that
$$\tp^{\cu M}(G_0(\omega)/C_0(\omega))=\tp^{\cu M}(A_0(\omega)/C_0(\omega))$$
and
$$ G_0(\omega) \cap B_0(\omega)\subseteq\dcl^{\cu M}(C_0(\omega)).$$
Let $i\le n$.  Whenever $a_i(\omega)\notin\dcl^{\cu M}(C_0(\omega))$, we have $g_i(\omega)\notin\dcl^{\cu M}(C_0(\omega))$, and hence
$g_i(\omega)\notin B_0(\omega)$. By Result \ref{f-glue}, for each $i\le n$ the event
$$\sa E_i=\l  a_i\in\dcl^{\cu M}(C_0)\rr$$
has a characteristic function $1_{{\sa E}_i}\in\cu K$ with respect to $\ti 0,\ti 1$.  By applying Condition (5) for a neat randomization to the formula
$$ \bigwedge_{\theta\in\Theta_0}(\theta(\vec u,\vec{\bo c})\leftrightarrow\theta(\vec{\bo a},\vec{\bo c}))\wedge
\bigwedge_{i=0}^n \bigwedge_{\bo b\in B_0} (1_{{\sa E}_i}=\ti 0\rightarrow u_i\ne\bo b),$$
we see that there exists a set
$$G_0=\{{\bo g}_0,\ldots,{\bo g}_n\}\subseteq \cu K$$
such that for each $\omega\in \Omega$, $\theta(\vec{\bo a},\vec{\bo c})\in\Theta_0$, $i\le n$, and $\bo b\in B_0$:
\begin{itemize}
\item
$\cu M\models \theta(\vec g(\omega),\vec c(\omega))\leftrightarrow\theta(\vec a(\omega),\vec c(\omega));$
\item if $a_i(\omega)\notin\dcl^{\cu M}(C_0(\omega))$, then $g_i(\omega)\ne b(\omega)$.
\end{itemize}
It follows that $\l\theta(\vec{\bo g},\vec {\bo c})\rr=\l\theta(\vec{\bo a},\vec {\bo c})\rr$ for each
$\theta(\vec{\bo a},\vec {\bo c})\in\Theta_0$, and that $d_\BK({\bo g}_i,{\bo b})\ge\varepsilon({\bo a}_i)$ for each $i\le n$ and $\bo b\in B_0$.
Therefore $\Gamma_0$ is satisfied by $G_0$ in $\cu N_{ABC}$, and the Claim is proved.

\

By saturation, $\Gamma$ is satisfied in $\cu N_{ABC}$ by some set $A'$.  $\Gamma$ guarantees that $A'\equiv_C A$ and  $\dcl_\BB(A'C)=\dcl_\BB(AC).$
It remains to show that for each $\bo a\in A$, $\bo a'\notin B$.  Let $\bo a\in A$.  By hypothesis
we have $\bo a\notin\dcl(C)$.
By Result \ref{f-dcl3}, either $\bo a$ is not pointwise definable over a countable subset of $C$ and thus $\varepsilon(\bo a)>0$,
or there is a formula $\theta(u,\vec v)$ and a tuple $\vec{\bo c}\in C^{<\BN}$ such that
$$\l\theta(\bo a,\vec{\bo c})\rr\in\fo_\BB(\{\bo a\}\cup C)\setminus\dcl_\BB(C).$$
$\Gamma$ guarantees that $d_\BK(\bo a',B)\ge\varepsilon(\bo a)$, so in the case that $\varepsilon(\bo a)>0$ we have $\bo a'\notin B$.
$\Gamma$ also guarantees that
$$\l\theta(\bo a',\vec{\bo c})\rr=\l\theta(\bo a,\vec{\bo c})\rr,$$
so in the case that $\varepsilon(\bo a)=0$, we have
$$\l\theta(\bo a',\vec{\bo c})\rr=\l\theta(\bo a,\vec{\bo c})\rr\in\dcl_\BB(AC)\setminus\dcl_\BB(C).$$
But we are assuming that
$$\dcl_\BB(AC)\cap\dcl_\BB(BC)=\dcl_\BB(C),$$
so
$$\l\theta(\bo a',\vec{\bo c})\rr\notin\dcl_\BB(BC),$$
and hence $\bo a'\notin B$.  This completes the proof.
\end{proof}

\begin{cor}  If $T$ has $\acl=\dcl$, then the relation $\ind[a]$ over $\cu N$ satisfies all the axioms for a countable independence relation
except perhaps base monotonicity.
\end{cor}

\begin{proof}  By Proposition \ref{p-alg-indep} and Theorem \ref{t-fe}.
\end{proof}

The next proposition shows that  $\ind[a]$ cannot be a countable independence relation over $\cu N$.

\begin{prop}  \label{p-a-nobase-monotonicity}  For every $T$, the relation
$\ind[a]$ over $\cu N$ does not have base monotonicity, and hence the lattice of algebraically closed sets in $\cu N$ is not modular.
\end{prop}

\begin{proof}   As in the proof of Proposition \ref{p-B-nobase-monotonicity}, we take two independent events $\sa D, \sa F$ in $\cu E$ of probability $1/2$,
and let $\sa E=\sa D\sqcap\sa F$. $\bo a=1_{\sa D}$, $\bo b = 1_{\sa E}$, and $\bo c=1_{\sa F}$.
Let $Z=\ti 0 \ti 1$.  Note that any element of $\cu K$ that is pointwise definable from $\bo a\bo b\bo c$ is pointwise definable from $Z$.
By Result \ref{f-dcl3} and the proof of Proposition \ref{p-B-nobase-monotonicity}, we have
$$ \dcl(\bo a Z)=\{\bo x\in\dcl^\omega(Z)\colon \fo_\BB(\bo x Z)\subseteq\sigma(\{\sa D\})\},$$
$$ \dcl(\bo c Z)=\{\bo x\in\dcl^\omega(Z)\colon \fo_\BB(\bo x Z)\subseteq\sigma(\{\sa F\})\},$$
$$ \dcl(\bo a \bo c Z)=\{\bo x\in\dcl^\omega(Z)\colon \fo_\BB(\bo x Z)\subseteq\sigma(\{\sa D,\sa F\})\},$$
$$ \dcl(\bo b \bo c Z)=\{\bo x\in\dcl^\omega(Z)\colon \fo_\BB(\bo x Z)\subseteq\sigma(\{\sa E,\sa F\})\}.$$
It follows that $\bo a\ind[a]_Z \bo b\bo c Z$ but $\bo a\nind[a]_{\bo c Z}  \bo b\bo c Z$, so $\ind[a]$ over $\cu N$ does not have base monotonicity.
\end{proof}

As an example, we look at the relations $\ind[a]$ and $\ind[M] \ $ in the continuous theory $\DLO^R$, the randomization of the theory of dense linear order without endpoints.
We will see that these relations are much more complicated in $\DLO^R$ than they are in $\DLO$.  This example is motivated by the open question \ref{q-DLO-rosy}.

\begin{ex}  Let $T=\DLO$, the theory of dense linear order without endpoints.  Over $\cu M$ we have $\acl(A)=\dcl(A)=A$ for every set $A$.  Thus in $\cu M$ the lattice
of algebraically closed sets is modular, and $\ind[a]=\ind[M]=\thind$.   But Proposition \ref{p-a-nobase-monotonicity} shows that in $\cu N$
the relation $\ind[a]$ does not have base monotonicity and hence $\ind[a]\ne\ind[M]$.  Proposition 4.2.3 of [AGK] shows that for every finite
set $A\subseteq\cu K$, $\dcl(A)$ is the smallest set $B\supseteq A$ such that whenever $\bo a, \bo b, \bo c, \bo d\in B$, the characteristic function
of $\l a < b\rr$ with respect to $\bo c, \bo d$ belongs to $B$.
Let $\bo a\vee\bo b$ and $\bo a\wedge\bo b$ denote the pointwise maximum and minimum, respectively.
We leave it to the reader to work out the following characterizations
of $A\ind[a]_C B$ and $A\ind[M]_C B$ in the simple case that $A, B, C$ are singletons in $\cu N$.
\begin{enumerate}
\item $\bo{a}\ind[M]_\emptyset \ \bo{b} \Leftrightarrow\bo{a}\ind[a]_\emptyset \bo{b}\Leftrightarrow\bo{a}\not=\bo{b}.$
\item  $\acl(\bo a\bo b)=\{\bo a,\bo b,\bo a\vee\bo b,\bo a\wedge\bo b\}$.
\item $\bo{a}\ind[a]_{\bo{c}}\bo{b}\Leftrightarrow
\{\bo{a},\bo{c},\bo{a}\vee\bo{c},\bo{a}\wedge\bo{c}\}\cap \{\bo{b},\bo{c},\bo{b}\vee\bo{c},\bo{b}\wedge\bo{c}\}=\{\bo{c}\}.$
\item If $\bo b\in\{\bo b\vee\bo c,\bo b\wedge\bo c\}$,  then $\bo{a}\ind[M]_{\bo{c}}\bo{b}\Leftrightarrow\bo{a}\ind[a]_{\bo{c}}\bo{b}$.
\item  If $\bo b\notin\{\bo b\vee\bo c,\bo b\wedge\bo c\}$, then $\bo{a}\ind[M]_{\bo{c}}\bo{b}$ if and only if:
\begin{itemize}
\item $\bo{a}\ind[a]_{\bo{c}}\bo{b}$, and
\item $\bo{b}\notin \dcl(\{\bo{a},\bo{c},\bo{b}\wedge\bo c\})$, and
\item $\bo{b}\notin \dcl(\{\bo{a},\bo{c},\bo{b}\vee\bo c\})$.
\end{itemize}
\end{enumerate}

Now take $\bo a,\bo b,\bo c$ such that
$$0<\mu(\l a=b\rr)=\mu(\l b<c \rr)<\mu(\l a<c\rr)<1$$
and
$$ \mu(\l a = c\rr)=\mu(\l  b= c\rr)=0,$$
and use (5) to show that $\bo a\ind[M]_{\bo c} \bo b$ but $\bo b\nind[M]_{\bo c} \bo a$.  Thus $\ind[a]$, $\ind[M] \ $, and $\thind$ are all different in the
big model $\cu N$ of $\DLO^R$.
\end{ex}

\section{Pointwise Independence}  \label{s-pointwise}

\subsection{The General Case}


\begin{df}  \label{d-pointwise}
If $\ind[I]$ is a ternary relation over $\cu M$ that has monotonicity, we let $\ind[I \omega]$ \quad
be the unique countably based relation over $\cu N$ such that for all countable $A,B,C$,
$$ A\ind[I \omega]_C \  B\Leftrightarrow A(\omega)\ind[I]_{C(\omega)} B(\omega) \ \as.$$
The unique existence of $\ind[I \omega]$ \quad follows from Lemma \ref{l-cbased} (2).
We say that \emph{$A$ is pointwise $I$-independent from $B$ over $C$} if $A\ind[I \omega]_C \ B$.
\end{df}

We will often use the notation $\l P\rr$ for the set $\{\omega\in\Omega\colon P(\omega)\}$
when $P(\omega)$ is a statement involving elements $\omega$ of $\Omega$.
Since $(\Omega,\cu F,\mu)$ is a complete probability space,  $P(\omega)$ holds  $\as$
if and only if $\mu(\l P \rr)=1$.
For instance, if $\ind[I]$ is a ternary relation over $\cu M$, then for all countable sets $A,B,C\subseteq\cu K$,
$$\l A\ind[I]_C B\rr=\{\omega\in\Omega \, : \, A(\omega)\ind[I]_{C(\omega)} \ B(\omega)\},$$
and
$$ A\ind[I \omega]_C \  B\Leftrightarrow\mu(\l A\ind[I]_C B\rr)=1.$$

\begin{cor} \label{c-pointwise-implies}
If $\ind[I]$ and $\ind[J]$ are  ternary relations over $\cu M$ with monotonicity, and $\ind[I]\Rightarrow\ind[J]$, then
$\ind[I \omega] \ \ \Rightarrow \ind[J \omega] \ \ $.
\end{cor}

\begin{proof}  This follows from Lemma \ref{l-cbased} (1).
\end{proof}

\begin{df}  A ternary relation $\ind[J]$ over $\cu N$ will be called \emph{pointwise anti-reflexive} if $\bo a\ind[J]_C \bo a$ implies $\bo a\in\acl^\omega(C)$.
\end{df}

Note that every anti-reflexive relation over $\cu N$ is pointwise anti-reflexive.

\begin{prop}  \label{p-omega-cbased}
Suppose $\ind[I]$ is a  ternary relation over $\cu M$ that has monotonicity.
\begin{enumerate}
\item  $\ind[I \omega]$ \ \ is countably based and has monotonicity and two-sided countable character.
\item If $\ind[I]$ has invariance, base monotonicity, transitivity, normality, symmetry, or the countable union property,
then  $\ind[I \omega]$ \ \ has the same property.
\item Suppose $\ind[I]$ has invariance.  If $\ind[I\omega] \ \ $ has base monotonicity, transitivity, normality, symmetry, or the countable union property,
then $\ind[I]$ has the same property.
\item If $\ind[I]$ is anti-reflexive, then $\ind[I\omega] \ \ $ is pointwise anti-reflexive.
\end{enumerate}
\end{prop}

\begin{proof}  (1) By definition, $\ind[I \omega]$ \ \ is countably based.  It has monotonicity and two-sided countable character by Lemma \ref{l-cbased} (3).

(2) By Lemma \ref{l-cbased} (4), it suffices to show that if $\ind[I]$ has one of the listed properties for countable sets, then
$\ind[I \omega]$ \ \ has the same property for countable sets.

 We prove the result for transitivity.  The other proofs are similar but easier.
 Suppose $\ind[I]$ has transitivity for countable sets, and assume that $A,B,C,D$
 are countable and $B\ind[I \omega]_C \  A$, $C\ind[I \omega]_D \  A$ and $C\in[D,B]$.
We must prove $B\ind[I \omega]_D \ A$. We have $\mu(\l B\ind[I]_C \ A\rr)=1$, $\mu(\l C\ind[I]_D \ A\rr)=1$, and $\mu(\l C\in [D,B]\, \rr)=1$.  Since
$\ind[I]$ has transitivity for countable sets, $\mu(\l B\ind[I]_D \ A\rr)=1$.  This shows that $B\ind[I \omega]_D  \ A$.

(3)  We again prove the result for transitivity.  Suppose $\ind[I\omega] \ \ $ has transitivity and let $A_0, B_0, C_0, D_0$ be countable subsets of $M$
such that $B_0\ind[I]_{C_0} \ A_0$ and $C_0\ind[I]_{D_0} \  A_0$ in $\cu M$, and $C_0\in[D_0,B_0]$.  Let $A, B, C, D$ be the images of $A_0, B_0, C_0, D_0$ in $\cu K$ under the
mapping $a\mapsto\ti a$.  Since $\ind[I]$ has invariance, it follows that $B\ind[I \omega]_C \  A$, $C\ind[I \omega]_D \  A$ and $C\in[D,B]$.
Since $\ind[I\omega] \ \ $ has transitivity, $B\ind[I \omega]_D \ A$.  Using invariance for $\ind[I]$ again, we have  $B_0\ind[I]_{D_0} \ A_0$.

(4)  Suppose $\ind[I]$ is anti-reflexive and $\bo a\ind[I\omega]_C \ \ \bo a$.  Then $\bo a\ind[I\omega]_D \ \ \bo a$ for some countable $D\subseteq C$,
so $a(\omega)\ind[I]_{D(\omega)} \ \ a(\omega) \as$, and hence $\bo a\in\acl^\omega(D)\subseteq\acl^\omega(C)$.
\end{proof}

To sum up, we have:

\begin{cor}  \label{c-summary}  If $\ind[I]$ satisfies the basic axioms for an independence relation over $\cu M$ ,
then $\ind[I \omega]$ \ \ over $\cu N$ also satisfies these axioms.
\end{cor}

\begin{prop}  \label{p-pointwise-finitechar}
Suppose $\ind[I]$ is a  ternary relation over $\cu M$ that is countably based and has finite character and the countable union property.
Then  $\ind[I \omega]$ \ \ has finite character.
\end{prop}

\begin{proof}  Suppose $A'\ind[I \omega]_C \ B$ for all finite $A'\subseteq A$.  Let $A_0\subseteq A, B_0\subseteq B, C_0\subseteq C$
be countable.  We must find a countable $D\in[C_0,C]$ such that $A_0\ind[I\omega]_{D} \ B_0$.  We may write $A_0=\bigcup_n E_n$ where
$E_0\subseteq E_1\subseteq\cdots$ and each $E_n$ is finite.  Then $E_n\ind[I \omega]_C B$ for each $n$.

Since $\ind[I \omega] \  $ is countably
based, there are countable sets $\<D_n\>_{n\in\BN}$ such that for each $n$, $D_n\in[C_0,C]$, $D_n\subseteq D_{n+1}$, and $E_n\ind[I \omega]_{D_n} \ B_0$,
and hence $E_n\ind[I]_{D_n} B_0 \ \as.$    Let $D=\bigcup_n D_n$.
Since $\ind[I]$ has the countable union property, for each $n$ we have $ E_n\ind[I]_{D} B_0 \ \as$.
Since $\ind[I]$ has finite character, $ A_0\ind[I]_{D} B_0 \ \as$., so $A_0\ind[I\omega]_D \ B_0$.
\end{proof}

\begin{df}  \label{d-measurable}
We say that a ternary relation $\ind[I]$ over $\cu M$ is \emph{measurable} if
 $\l A\ind[I]_C B\rr\in\cu F$ for all countable $A,B,C\subseteq\cu K$.
\end{df}

Measurability will be useful in showing that particular pointwise relations satisfy countable versions of full existence.
We will sometimes use measurability without explicit mention in the following way:  if $\ind[I]$ is measurable and $A \nind[I]_C B$, then $\mu(\l A\ind[I]_C B\rr)=r$ for some $r<1$.

Our next lemma gives a useful sufficient condition for measurability.
 $L_{\omega_1\omega}$ is the infinitary logic that contains first order logic and is closed under countable conjunctions and disjunctions, negations,
and finite existential and universal quantifiers.   An $L_{\omega_1\omega}$ formula is said to be \emph{conjunctive}  if it is built from first order
formulas using only countable conjunctions, and finite conjunctions, disjunctions, and quantifiers.  By a \emph{Borel-conjunctive} formula we mean an
$L_{\omega_1\omega}$ formula that is built from conjunctive formulas using only negations and finite and countable conjunctions and disjunctions.

The following result is a consequence of Theorem 2.3 in [Ke2].

\begin{result}  \label{f-conjunctive}
Suppose  $\cu M$ is an $\aleph_1$-saturated first order structure.  For every countable set $X$ of variables and conjunctive $L_{\omega_1\omega}$-formula $\theta(X)$
there is a countable set $\cu A(\theta)$ of first order formulas (the set of finite approximations of $\theta$) such that
$$\cu M\models (\forall X)\left[\theta(X)\leftrightarrow\bigwedge\left\{\psi(X)\colon\psi\in\cu A(\theta)\right\}\right].$$
\end{result}

We say that $\ind[I]$ is \emph{definable by} an $L_{\omega_1\omega}$  formula $\varphi(X,Y,Z)$ in $\cu M$ with countable sets of variables $X,Y,Z$ if for all countable sets
$A, B, C$ indexed by $(X,Y,Z)$ in $\cu M$, we have
$$A\ind[I]_C B\Leftrightarrow\cu M\models\varphi(A,B,C).$$

\begin{lemma}  \label{l-conjunctive-measurable}
Suppose $\ind[I]$ is definable by a Borel-conjunctive formula $\varphi$ in $\cu M$.  Then $\ind[I]$ is measurable, and for all countable sets $A, B, C\subseteq\cu K$,
the reduction of $\l A\ind[I]_C B\rr$ belongs to $\dcl_\BB(ABC)$.
\end{lemma}

\begin{proof}  By Result \ref{f-conjunctive}, for every conjunctive formula $\theta(X,Y,Z)$ and all countable
sets $A,B,C\subseteq\cu K$ we have
$$\l\theta(A,B,C)\rr=\bigcap\{\l\psi(A,B,C)\rr\colon\psi\in\cu A(\theta)\}.$$
By definition, $\l\psi(A,B,C)\rr\in\cu F$ for every first order formula $\psi$.  Since $\cu F$ is a $\sigma$-algebra,
it follows that $\l\theta[A,B,C)\rr\in\cu F$ for every conjunctive formula $\theta$.  Since $\varphi$ is Borel-conjunctive, we have
$$\l A\ind[I]_C B\rr=\l\varphi(A,B,C)\rr\in\cu F.$$
Because the reduction of $\l\psi(A,B,C)\rr$ belongs to $\fo_\BB(ABC)$ for each first order formula $\psi$, and $\mu$ is $\sigma$-additive,
we see from Result \ref{f-separable} that the reduction of $\l\varphi(A,B,C)\rr$ belongs to $\dcl_\BB(ABC)$.
\end{proof}

\subsection{Pointwise Algebraic Independence}


In this section we will show that the relation $\ind[a\omega] \ \ $ has various independence properties.
Since $\ind[a]$ over $\cu M$ has monotonicity, $\ind[a\omega] \ \ $ over  $\cu N$ exists.

\begin{df} We call a first order formula $\theta(u,X)$ \emph{algebraical (with bound $n$)} if
$$T\models(\forall X)(\exists^{\le n} u)\theta(u,X).$$
If $\theta(u,X)$ is algebraical and $A$ is indexed by $X$, we say that $\theta(u,A)$ is algebraical over $A$.
\end{df}

Note that in $\cu M$, we have $b\in\acl(A)$ if and only if $b$ satisfies some algebraical formula over $A$ in $\cu M$.

\begin{cor}  \label{c-aa-omega-easy}  The relation
$\ind[a\omega] \ \ $  over $\cu N$ satisfies invariance, monotonicity, transitivity, normality, finite character, symmetry, the countable union property,
and pointwise anti-reflexivity. Moreover, $\ind[a\omega] \ $ has base monotonicity if and only if $\ind[a]$ has base monotonicity in $\cu M$
(and thus if and only if the lattice of algebraically closed sets in $\cu M$ is modular).
\end{cor}

\begin{proof}  The relation $\ind[a]$ over $\cu M$ satisfies invariance, monotonicity, transitivity, normality, finite character, symmetry,
the countable union property, and anti-reflexivity.  Now apply
Proposition \ref{p-omega-cbased}, and Proposition \ref{p-pointwise-finitechar} for finite character.
\end{proof}

\begin{prop} \label{p-ABB}
For all small $A, B$ in $\cu K$, we have  $A\ind[a\omega]_B \ B$.
\end{prop}

\begin{proof}
The relation $\ind[a\omega] \ \ $ is countably based by definition, so to prove $A\ind[a\omega]_B \ B$ it suffices to show that
$$ (\forall^c A'\subseteq A)(\forall^c B'\subseteq B)(\forall^c C'\subseteq B)(\exists^c D\in[C',B])A'\ind[a\omega]_{D} \ B'.$$
Let $D=B'\cup C'$.  Then $D$ is countable and $D\in[C',B]$.   We show that $A'\ind[a\omega]_{D} \ B'$.

For every $\omega\in\Omega$, we have
$$\acl(A'(\omega)D(\omega))\cap\acl(B'(\omega)D(\omega))=\acl(A'(\omega)D(\omega))\cap\acl(D(\omega))=\acl(D(\omega)),$$
so  $A'(\omega)\ind[a]_{D(\omega)} \ D(\omega)$.  By monotonicity, $A'(\omega)\ind[a]_{D(\omega)} \ B'(\omega)$ for all $\omega$, and hence
$A'\ind[a\omega]_{D} \ B'$ as required.
\end{proof}

\begin{lemma}  \label{l-a-measurable}  The relation $\ind[a]$ on the first order theory $T$ is measurable.
\end{lemma}

\begin{proof}  Let $\varphi_i(u,X,Z)$, $\psi_i(u,Y,Z)$, and $\chi_i(u,Z)$ enumerate all algebraical formulas over the indicated variables.
Then $\ind[a]$ is definable in $\cu M$ by the Borel-conjunctive formula
$$\neg\bigvee_i\bigvee_j(\exists u)[\varphi_i(u,X,Z)\wedge \psi_j(u,Y,Z) \wedge \bigwedge_k \neg\chi_k(u,Z)].$$
Hence by Lemma \ref{l-conjunctive-measurable}, $\ind[a]$ is measurable.
\end{proof}

\begin{thm}  \label{t-aa-omega-existence}
The  relation $\ind[a\omega] \ \ $ over $\cu N$ satisfies extension and  full existence for all countable sets $A, B, \widehat{B}, C$..
\end{thm}

\begin{proof}
We first prove  full existence for countable sets.  Let $A, B, C$ be countable subsets of $\cu K$.  By Proposition \ref{p-event-aB}, the relation
$\ind[a\BB] \ \ $ over $\cu N$ has full existence.  Therefore we may assume that $ A\ind[a\BB]_C \  B.$  By Result \ref{f-acl=dcl},
$$\dcl_\BB(AC)\cap\dcl_\BB(BC)=\dcl_\BB(C).$$
Since $\ind[a]$ has full existence in $\cu M$,
for each $\omega\in\Omega$ there exists a set $A'_0\subseteq M$ such that $A'_0\equiv_{C(\omega)} A(\omega)$ and $A'_0\ind[a]_{C(\omega)} \ B(\omega)$
in $\cu M$.

Let $\varphi_i(u,A,C)$, $\psi_i(u,B,C)$, and $\chi_i(u,C)$ be enumerations of all algebraical formulas
over the indicated sets (with repetitions) such that for each pair of algebraical formulas $\varphi(u,A,C)$ and $\psi(u,B,C)$ there
exists an $i$ such that $(\varphi_i,\psi_i)=(\varphi,\psi)$.  Whenever $\omega\in\Omega$, $A'_0\subseteq\cu M$, and $A'_0\ind[a]_{C(\omega)} \ B(\omega)$ in $\cu M$,
for each $i\in \BN$ there exists $j\in\BN$ such that
\begin{equation} \label{eq-ind}
 \cu M \models \forall u [\varphi_i(u,A'_0,C(\omega))\wedge\psi_i(u,B(\omega),C(\omega))\rightarrow\chi_j(u,C(\omega))].
 \end{equation}
Let $\BN^0=\{\emptyset\}$ and $\sa E_\emptyset=\Omega$.
For each $n>0$ and $n$-tuple $s=\langle s(0),\ldots,s(n-1)\rangle$ in $\BN^n$, let $\sa E_s$ be the set of all $\omega\in\Omega$ such that
for some set $A'_0\subseteq\cu M$, $A'_0\equiv_{C(\omega)} B(\omega)$ and (\ref{eq-ind}) holds whenever $i<n$ and $j=s(i)$.

Let $L'$ be the signature formed by adding to $L$ the constant symbols
$$\{ k_a, k_b, k_c \ : \ a\in A, b\in B, c\in C\}.$$
For each $\omega\in\Omega$, $(\cu M,A(\omega),B(\omega),C(\omega))$ will be the $L'$-structure where $k_a, k_b, k_c$ are interpreted by
$a(\omega), b(\omega), c(\omega)$.  Form $L''$ by adding to $L'$ countably many additional constant symbols $\{k'_a \, : \, a\in A\}$ that will be used for elements of
a countable subset $A'_0$ of $\cu M$.

Then for each $n>0$ and $s\in \BN^n$, there is a countable set of sentences
$\Gamma_s$ of $L''$ such that for each $\omega$, $\omega\in \sa E_s$ if and only if $\Gamma_s$ is satisfiable in $(\cu M,A(\omega),B(\omega),C(\omega))$.
Since $\cu M$ is $\aleph_1$-saturated, $\Gamma_s$ is satisfiable if and only if it is finitely satisfiable in $(\cu M,A(\omega),B(\omega),C(\omega))$.
It follows that the set $\sa E_s$ belongs to the $\sigma$-algebra $\cu F$.  Moreover, since $\ind[a]$ has full existence in $\cu M$, for each $n$ and $s\in \BN^n$ we have
$$\Omega\doteq\bigcup\{\sa E_t \colon t\in \BN^n\},\quad  \sa E_s \doteq\bigcup\{ \sa E_{sk}\colon k\in\BN\},$$
where $sk$ is the $(n+1)$-tuple formed by adding $k$ to the end of $s$.  We now cut down the sets $\sa E_s$ to sets $\sa F_s\in\cu F$ such that:
\begin{itemize}
\item[(a)] $\sa F_\emptyset=\Omega$;
\item[(b)] $\sa F_s\subseteq\sa E_s$ whenever $s\in \BN^n$;
\item[(c)] $\sa F_s\cap\sa F_t=\emptyset$ whenever $s,t\in \BN^n$ and $s\ne t$;
\item[(d)] $\sa F_s\doteq\bigcup\{\sa F_{sk}\colon k\in\BN\}$ whenever $s\in\BN^n$.
\end{itemize}
This can be done as follows.  Assuming $\sa F_s$ has been defined for each $s\in\BN^n$. we let
$$ \sa F_{sk}=\sa F_s\cap(\sa E_{sk}\setminus\bigcup_{j<k} \sa F_{sj}).$$
Now let $\theta_i(A,C)$ enumerate all first order sentences with constants for the elements of $AC$.  Let $\Sigma$ and $\Delta$ be the following countable sets of sentences of $(L'')^R$:
$$\Sigma=\{\l \theta_i(A',C)\rr\doteq\l\theta_i(A,C)\rr\colon i\in\BN\}.$$
$$\Delta=\{\sa F_s\sqsubseteq \l\forall u[\varphi_i(u,A',C))\wedge\psi_i(u, B, C))\rightarrow \chi_{s(i)}(u,C))]\rr\colon s\in\BN^{<\BN}, i<|s|\}.$$
It follows from Fullness and conditions (a)--(d) above that $\Sigma\cup\Delta$ is finitely satisfiable in $\cu N_{ABC}$.  Then by saturation, there is a set $A'$
that satisfies $\Sigma\cup\Delta$ in $\cu N_{ABC}$.  Since $A'$ satisfies $\Sigma$, we have $A'\equiv_C A$.
The sentences $\Delta$ guarantee that $A'\ind[a\omega]_C \ B$.

By the proof of Remarks \ref{r-fe-vs-ext} (1) (see the Appendix of [Ad1]), invariance, monotonicity, transitivity, normality,  symmetry,
and full existence for all countable sets implies extension for all countable sets.  Then by the preceding paragraphs and Corollary \ref{c-aa-omega-easy},
$\ind[a\omega] \ \ $ satisfies  extension for all countable sets.
\end{proof}

\begin{question}  Does $\ind[a\omega] \ \ $ satisfy extension for countable $A,B,C$ and small $\widehat{B}$?
\end{question}

\begin{question}  Does $\ind[a\omega] \ \ $ satisfy full existence and/or extension?
\end{question}

We next look for connections between the relations $\ind[a]$ and $\ind[a\omega] \ \ $ over $\cu N$.

\begin{prop}  \label{p-a-vs-omega}
Let $\ind[I]=\ind[a\omega] \ \ \wedge\ind[a\BB] \ $ on $\cu N$.
\noindent\begin{enumerate}
\item $\ind[I] $ is countably based.
\item If $\ind[I] \Rightarrow \ind[a]$, or even if $\ind[I]$
is anti-reflexive, then $T$ has $\acl=\dcl$.
\item If $T$ has $\acl=\dcl$ then $\ind[I] \Rightarrow \ind[a]$.
\end{enumerate}
\end{prop}

\begin{proof}
(1) By Propositions \ref{p-union-pair}, \ref{p-event-aB} and \ref{p-a-B-cbased}, and Corollary \ref{c-aa-omega-easy}.

(2) Suppose  that $T$ does not have $\acl=\dcl$.  Then in $\cu M$ there is a finite set $C$ and an element $a\in\acl(C)\setminus\dcl(C)$.
Therefore in $\cu N$ we have $\ti a\ind[a\omega]_{\ti C} \ \ \ti a \ \wedge \ti a\ind[a\BB]_{\ti C} \ \ti a$, but $\ti a\notin\acl(\ti C)$ and $\ti a \nind[a]_{\ti C} \ti a$.

(3)  Suppose $T$ has $\acl=\dcl$, $A\ind[a\omega]_C \ B$, and $A\ind[a\BB]_C \ B$.  To prove $A\ind[a]_C B$, it suffices to show that
$$\dcl(AC)\cap\dcl(BC)\subseteq\dcl(C).$$
By (1), Proposition \ref{p-a-cbased}, and Lemma \ref{l-cbased} (1), we may assume that $A, B, C$ are countable. Let
$$\bo d\in\dcl(AC)\cap\dcl(BC).$$
By Result \ref{f-dcl3}, $\dcl_\BB({\bo d}AC)\subseteq\dcl_\BB(AC)$ and $\dcl_\BB({\bo d}BC)\subseteq\dcl_\BB(BC)$. Since $A\ind[a\BB]_C \ B$, it follows that
$$ \dcl_\BB({\bo d} C)\subseteq \dcl_\BB(AC)\cap\dcl_\BB(BC)\subseteq\dcl_\BB(C).$$
Result \ref{f-dcl3} also gives
$$\bo d\in\dcl^\omega(AC)\cap\dcl^\omega(BC).$$
Since $A\ind[a\omega]_C \ B$ and $T$ has $\acl=\dcl$,
$$d(\omega)\in\acl(AC(\omega))\cap \acl(BC(\omega))=\acl(C(\omega))=\dcl(C(\omega)) \as.$$
Hence $\bo d\in\dcl^\omega(C)$, and by Result \ref{f-dcl3} in the other direction, $\bo d\in\dcl(C)$.
\end{proof}

\begin{prop}  \label{p-pointwise-algebraic-vs-algebraic}

\noindent\begin{enumerate}
\item $\ind[a\omega] \ \Rightarrow \ind[a]$ always fails in $\cu N$.
\item $\ind[a]\Rightarrow\ind[a\omega] \ \ $ holds in $\cu N$ if and only if the models of $T$ are finite.
\end{enumerate}
\end{prop}

\begin{proof}  (1)  Let $\sa D\in\cu E$ be an event such that $0<\mu(\sa D)<1$, and let $C=\ti 0 \ti 1$.  Then $1_{\sa D}\in\dcl^\omega(C)\setminus\dcl(C)$.
Hence $1_{\sa D}\ind[a\omega]_C \ 1_{\sa D}$ but $1_{\sa D}\nind[a]_C 1_{\sa D}$, so $\ind[a\omega] \ \Rightarrow \ind[a]$ fails.

(2)  If $\cu M$ is finite, then $\acl^{\cu M}(\emptyset)=M$, so $A\ind[a]_C B$ always holds in $\cu M$.  Therefore $A\ind[a\omega]_C \ B \ $ always holds in $\cu N$,
and hence $\ind[a]\Rightarrow\ind[a\omega] \ \ $ holds in $\cu N$.

For the other direction, assume $\cu M$ is infinite.  By saturation, $\cu M$ has elements $0,1,a,b$ such that
$$ 0\ne 1,\quad a\notin\acl(01), \quad\tp(a/\acl(01))=\tp(b/\acl(01)), \quad a\ind[a]_{01} \, b.$$
We will use the mapping $a\mapsto\ti a$ from Lemma \ref{l-big-embedding}.
To simplify notation, suppose first that $T$ already has a constant symbol for each element of $\acl(01)$.  Then
$\acl(01)=\acl(\emptyset)$ in $\cu M$, so
$$ 0\ne 1, \quad a\notin\acl(\emptyset),\quad\tp(a)=\tp(b), \quad a\ind[a]_{\emptyset} \, b \quad\mbox{ in } \cu M,$$
$$ \mu(\l {\ti 0}\ne {\ti 1}\rr)=1,\quad {\ti a}\notin\dcl(\emptyset),\quad\tp({\ti a})=\tp({\ti b}), \quad {\ti a}\ind[a]_{\emptyset} \, {\ti b} \quad\mbox{ in } \cu N.$$
By Results \ref{f-acl=dcl} and \ref{f-separable}, for each $A\subseteq \ti M$,
$$\acl(A)=\dcl(A)=\cl(\fo(A))=\fo(A)\subseteq \ti M.$$

Let $\sa E\in\cu E$ be an event of measure $\mu(\sa E)=1/2$.
Let $ \bo c=1_{\sa E}$, let $\bo d$ agree with $\ti a$ on $\neg \sa E$ and with $\ti b$ on $\sa E$, and let $\bo e$  agree with $\ti 1$ on $\sa E$
and with $\ti a$ on $\neg \sa E$ (see the figure).
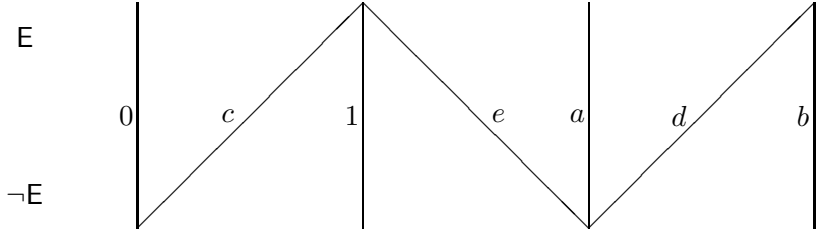
\begin{center}
\setlength{\unitlength}{1.5mm}
\begin{picture}(80,20)
\put(10,0){\line(0,1){20}}
\put(30,0){\line(0,1){20}}
\put(50,0){\line(0,1){20}}
\put(70,0){\line(0,1){20}}
\put(10,0){\line(1,1){20}}
\put(50,0){\line(1,1){20}}
\put(30,20){\line(1,-1){20}}

\put(0,3){\makebox(0,0){$\neg\sa E$}}
\put(0,17){\makebox(0,0){$\sa E$}}
\put(9,10){\makebox(0,0){$0$}}
\put(29,10){\makebox(0,0){$1$}}
\put(49,10){\makebox(0,0){$a$}}
\put(69,10){\makebox(0,0){$b$}}
\put(18,10){\makebox(0,0){$c$}}
\put(42,10){\makebox(0,0){$e$}}
\put(58,10){\makebox(0,0){$d$}}
\end{picture}
\end{center}

\emph{Claim 1}: $\ti a\ind[a]_{\emptyset} \bo{cd} $ in $\cu N$.

\emph{Proof of Claim 1}:  Suppose $ x\in\acl({\ti a})\cap\acl(\bo{cd})$ in $\cu N$.  Then $x\in \dcl(\ti a)$,
so $x=\ti z$ for some $z\in\dcl^{\cu M}(a)$.
Moreover, $x\in\dcl(\bo c\bo d)$, so $x\in\dcl^\omega(\bo c\bo d)$, and hence $x(\omega)\in\dcl^{\cu M}(1b)=\dcl^{\cu M}(b)$ for all $\omega\in\sa E$.
Therefore $z\in\dcl^{\cu M}(b)$.  Since ${\ti a}\ind[a]_{\emptyset} \, {\ti b}$ in $\cu N$, we have
$x\in\acl(\ti a)\cap\acl(\ti b)=\acl(\emptyset)$.
\medskip

\emph{Claim 2}: $\ti a\nind[a\omega]_{\emptyset} \ \bo{cd} $ in $\cu N$.

\emph{Proof of Claim 2}:
For all $\omega\in\neg\sa E$ we have $\ti a(\omega)= a, \bo c(\omega)=0, \bo d(\omega)=a$.  Hence $a\in\acl(\ti a(\omega)\cap \bo c\bo d(\omega))\setminus\acl(\emptyset)$
and $\omega\notin\l \ti a\ind[a]_\emptyset \bo c\bo d\rr$.  Therefore $\mu(\l \ti a\ind[a]_\emptyset\bo c\bo d\rr)\le 1/2$, so $\ti a\nind[a\omega]_{\emptyset} \ \bo{cd} $.

By Claims 1 and 2, $\ind[a]\Rightarrow\ind[a\omega] \ \ $ fails in $\cu N$.

We now turn to the general case where $T$ need not have a constant symbol for each element of $\acl(01)$.
Let $Z=\acl(01)$ in $\cu M$.
Our argument above shows that $\ti a\ind[a]_{\ti Z} \bo{cd}$ but $\ti a\nind[a\omega]_{\ti Z} \ \bo{cd}$ in $\cu N$, so
$\ind[a]\Rightarrow\ind[a\omega] \ \ $ still fails in $\cu N$.
\end{proof}

\subsection{Pointwise $M$-independence}


We now consider the relation $\ind[M \omega] \ \  $ of pointwise $M$-independence.  The relation $\ind[M] \ $ is monotone over $\cu M$, so $\ind[M\omega] \ \ $ over $\cu N$ exists.

\begin{cor} \label{c-M-omega-easyaxioms}
 The relation $\ind[M \omega] \ \ \ $ over $\cu N$ is countably based,  and
satisfies all the axioms for a countable independence relation except perhaps extension and local character.
Moreover, if $\ind[M]$ \ on models of $T$ has symmetry then  $\ind[M \omega]$ \quad over $\cu N$ has symmetry.
\end{cor}

\begin{proof}  This follows from Proposition \ref{p-omega-cbased} and Corollary \ref{c-summary}.
\end{proof}

\begin{cor}  \label{c-M-omega-implies-a-omega}
$\ind[M\omega] \ \ \ \Rightarrow \ind[a\omega] \ $, and $\ind[M\omega] \ \ \ $ is pointwise anti-reflexive.
\end{cor}

\begin{proof}  This follows from Corollary \ref{c-pointwise-implies}, and the fact that
$\ind[M] \ \Rightarrow\ind[a]$.
\end{proof}


\begin{lemma}  \label{l-M-union}
For each complete first order theory $T$, the relation $\ind[M] \ $ has the countable union property.
\end{lemma}

\begin{proof}  Suppose that $A, B, C$ are countable sets in $\cu M$, $C=\bigcup_n C_n$, and
$C_n\subseteq C_{n+1}$ and $A\ind[M]_{C_n} B$ for each $n$. We must show that $A\ind[M]_C B$.  Let $D\in[C,\acl(BC)]$ and $x\in\acl(AD)\cap\acl(BD)$.
For each $n$, let $D_n=D\cap\acl(BC_n)$.  Then $D_0\subseteq D_1\subseteq\cdots$ and $D=\bigcup_n D_n$, so
$x\in\acl(AD_n)\cap\acl(BD_n)$ for some $n$.  We have $D_n\in[C_n,\acl(BC_n)]$ and $A\ind[M]_{C_n} B$.  Therefore $x\in\acl(D_n)\subseteq\acl(D)$,
so $A\ind[a]_D B$ and $A\ind[M]_C B$.
\end{proof}

\begin{cor}  \label{c-M-omega-finitechar}
The relation $\ind[M\omega] \ \ $ over $\cu N$ has finite character.
\end{cor}

\begin{proof}
It is shown in [Ad2] that the relation $\ind[M] \ $ over $\cu M$ has monotonicity and finite character.
$\ind[M] \ $ over $\cu M$ also has the countable union property by Lemma \ref{l-M-union}.  Hence, by Proposition \ref{p-pointwise-finitechar},
$\ind[M\omega] \ \ \ $ has finite character.
\end{proof}

\begin{lemma}  \label{l-M-measurable}
 The relation $\ind[M] \  $ over $\cu M$ is measurable.
\end{lemma}

\begin{proof}   For each $n\in\BN$, let $\vec v_n$ denote the $n$-tuple of variables $\<v_1,\ldots,v_n\>$.  Let $X,Y,Z$ be countable sets of variables.
Let $\eta_i(u,Y,Z)$ enumerate all algebraical formulas over $Y,Z$.  Let $\varphi_i^n(u,\vec v_n,X,Y,Z)$ enumerate all formulas
of the following form:
$$\bigwedge_{j=1}^n \eta_{i_j}(v_j,Y,Z)\wedge \psi(u,\vec v_n,X,Z)\wedge \chi(u,\vec v_n,Y,Z),$$
where $\psi$ and $\chi$ are also algebraical formulas; here $n$ is allowed to vary.  Let $\theta_k^n(u,\vec v_n,Z)$ enumerate all algebraical
formulas over $\vec v_n,Z$.  Then the relation $\ind[M] \ $ over $\cu M$ is definable by the Borel-conjunctive formula
$$\cu M \models
\neg\bigvee_n\bigvee_i(\exists u)(\exists \vec v_n)\left[\varphi_i^n(u,\vec v_n,X,Y,Z)\wedge \bigwedge_k\neg \theta_k^n(u,\vec v_n,Z))\right].$$
Thus $\ind[M] \ $ is measurable by Lemma \ref{l-conjunctive-measurable}.
\end{proof}

The next result gives a characterization of the relation $\ind[M\omega] \ \ \ $.

\begin{prop}  \label{p-M-omega-char}  $\ind[M\omega] \ \ \ $ is the unique countably based relation such that
 for all countable $A,B,C$,
 $$ A\ind[M \omega]_C \  \ B \Leftrightarrow (\forall^c D\in[C,\acl^\omega(BC)]) A\ind[a \omega]_D \  B.$$
\end{prop}

\begin{proof}
Assume first that $ A\ind[M\omega]_C \  B $.  Suppose $D$ is countable and $D\in[C,\acl^\omega(BC)]$.  Then
$A(\omega)\ind[M]_{C(\omega)} \  B(\omega) \as.$ and $C(\omega)\subseteq D(\omega)\as$.  For each $\bo d\in D$ we have
 $\bo{d}(\omega)\in \acl(B(\omega)C(\omega))\as$.  Since $D$ is countable, it follows that
 $D(\omega)\subseteq\acl(B(\omega)C(\omega)) \as$.  Therefore $A(\omega)\ind[a]_{D(\omega)} B(\omega) \as.$, so $A\ind[a \omega]_D \  B.$

For the other direction, assume that $A\nind[M \omega]_C \ \ B$. If $\acl^\omega(BC)=\emptyset$, then $\mu(\l \acl(BC)=\emptyset\rr)=1$,
so we would trivially have $A\ind[M \omega]_C \ \ B$.  Therefore $\acl^\omega(BC)$ is non-empty, and we may  take $\bo b\in\acl^\omega(BC)$.
Work with the notation from Lemma \ref{l-M-measurable}.  There is $n,i$ such that
$$\sa E:=\bigcap_k\l (\exists u)(\exists \vec v_n)(\varphi_i^n(u,\vec v_n,ABC)\wedge \neg \theta_k^n(u,\vec v_n,C))\rr$$
has positive measure, where $\varphi_i^n(u,\vec v_n,ABC)$ has the form
$$\bigwedge_{j=1}^n \eta_{i_j}(v_j,BC)\wedge \psi(u,\vec v_n,AC)\wedge \chi(u,\vec v_n,BC)$$
and $\eta_{i_j}(v_j,BC)$, $\psi(u,\vec v_n,AC)$, $\chi(u,\vec v_n,BC)$, and $\theta_k^n(u,\vec v_n,C))$ are algebraical.
Take a formula $\tau(\vec z)$ and a tuple $\vec{\bf e}\in \cu K$ such that $\sa E=\l \tau(\vec e)\rr$.  Then
$$\mu(\l (\exists u)(\exists\vec v_n)[(\tau(\vec e)\to (\varphi_i^n(u,\vec v_n,ABC)\wedge \neg \theta_k^n(u,\vec v_n,C))$$
$$\wedge(\neg \tau(\vec e)\to \bigwedge_j v_j=b]\rr)=1.$$
By fullness there is an $n+1$-tuple $(\bo a,\vec {\bo d})$ witnessing the quantifiers $(\exists u)(\exists \vec v_n)$ in the above formula.
Notice that each $d_j\in \acl^\omega(BC)$.  Set
$$D=C\cup\{\vec d_n\}\in [C,\acl^\omega(BC)].$$
Then $D$ is countable and
$$\sa E\sqsubseteq\l a\in [\acl(AD)\cap \acl(BC)]\setminus \acl(D)\rr\sqsubseteq\l A\nind[a]_D B\rr.$$
Therefore $A\nind[a \omega]_D\ B$, and the proof is complete.
\end{proof}

\subsection{Pointwise Dividing Independence}

In this subsection we consider the relation $\ind[d\omega] \ \ $, especially when the first order theory $T$ is simple.
The relation $\ind[d]$ over $\cu M$ has monotonicity, so $\ind[d\omega] \ \ $ exists.


\begin{cor}  \label{c-simple-easyaxioms}
Suppose $T$ is a simple first order theory.  Then the relation $\ind[d\omega] \ $ over $\cu N$
satisfies the basic axioms, and has symmetry and finite character.
\end{cor}

\begin{proof}  The basic axioms and symmetry follow from Proposition \ref{p-omega-cbased} and Corollary \ref{c-summary}.
By Proposition \ref{p-d-cbased}, $\ind[d]$ over $\cu M$ is countably based.  Then by Lemma \ref{l-d-union} and Proposition \ref{p-pointwise-finitechar},
$\ind[d\omega] \ \ $ has finite character.
\end{proof}

\begin{prop}  \label{p-ind[d]-meas}
For every complete first order theory $T$, the relation $\ind[d]$ over $\cu M$ is measurable.
\end{prop}

\begin{proof}
Let $\varphi(\vec x,\vec y,Z)$ be a first order formula, where
$\vec x,\vec y$ are tuples of variables, and $Z$ is a countable set of variables.
For each tuple $\vec b$ and countable set $C$ in the big first order model $\cu M$, $\varphi(\vec x, \vec b, C)$ divides over $C$
if and only if  $(\cu M,\vec b, C)$ satisfies the following Borel-conjunctive formula $\div_{\varphi}(\vec y,Z)$:
$$ \bigvee_k\bigwedge_{n\ge k}(\exists \vec y^{\, 0},\ldots,\vec y^{\, n-1})
\left[\bigwedge_{j<n}\vec y^{\, j}\equiv_Z \vec y
\wedge \bigwedge_{I\subset n, |I|=k}\neg(\exists\vec x)\bigwedge_{i\in I} \varphi(\vec x,\vec y^{\, i},Z)\right].$$
Therefore $\ind[d]$ is definable in $\cu M$ by the Borel-conjunctive formula
$$ \neg\bigvee_{\vec x\in X^{<\BN}}\bigvee_{\vec y\in Y^{<\BN}}\bigvee_\varphi (\varphi(\vec x,\vec y,Z)\wedge \div_\varphi(\vec y,Z)),$$
where $X,Y,Z$ are used to index $A,B,C$.  So by Lemma \ref{l-conjunctive-measurable},  $\ind[d]$ is measurable.
$\ind[d]$ also has monotonicity by Result \ref{f-d-indep}, so $\ind[d\omega] \ $ exists.
\end{proof}

\begin{prop}  \label{p-simple-existence}
Suppose that $T$ is simple.  Then for any small $A,B$ and countable $C$, there is $A'\equiv_C A$ such that $A'\ind[d\omega]_C \ B  $.
\end{prop}

\begin{proof}
Let $(A_\alpha,B_\alpha)$ enumerate all pairs of finite subsets of $A$ and $B$.  Let $X$ be a set of variables corresponding to the elements of $A$,
and let $X_\alpha$ correspond to $A_\alpha$.  From the proof of Proposition \ref{p-ind[d]-meas},
for each $\alpha$ and first order formula $\varphi(X_\alpha,B_\alpha,C)$, we have $\l \div_{\varphi}(B_\alpha,C)\rr\in\cu F$.
We show that the following set of statements $\Gamma(X)$ in satisfiable in $\cu N$:
\begin{itemize}
\item $\l \theta(X,C)\rr\doteq\l \theta(A,C)\rr$ for each first order formula $\theta$;
\item $ \l \div_{\varphi}(B_\alpha,C)\rr\sqsubseteq\l \neg \varphi(X_\alpha,B_\alpha,C)\rr$ for each $\alpha$ and $\varphi$.
\end{itemize}

Since $T$ is simple, forking coincides with dividing, so for each $\omega\in\Omega$, every finite disjunction of formulas that divide over $C(\omega)$ also divides over $C(\omega)$.
Moreover,  $\ind[d]$ on models of $T$ satisfies existence. It follows from fullness that each finite subset of $\Gamma$ is satisfiable in $\cu N$.
By saturation, $\Gamma(X)$ is satisfied in $\cu N$ by some set $A'$.  Then $A'\equiv_C A$.
Let $A''\subseteq A'$ and $B''\subseteq B$ be countable.  Since $C$ is countable, to show $A'\ind[d\omega]_C \ B$
it suffices to show that $A''\ind[d\omega]_C B'' \ $.  The proof of Proposition \ref{p-ind[d]-meas} shows that
$$\l A''(\omega)\nind[d]_{C(\omega)}B''(\omega)\rr \sqsubseteq \bigcup_\alpha \bigcup_\varphi (  \l \div_{\varphi}(B_\alpha,C)\rr\cap \l\varphi(A'_\alpha,B_\alpha,C)\rr),$$
where $\alpha$ is such that
$A_\alpha\subseteq A''$ and $B_\alpha\subseteq B''$.  This union has measure $0$ because $A'$ satisfies $\Gamma(X)$, so $A''\ind[d\omega]_C B'' \ $.
\end{proof}

\begin{cor}  \label{c-simple-existence}
If $T$ is simple, then $\ind[d\omega] \ \ $ satisfies full existence when $C$ is countable, and satisfies extension
when $B, C$ are countable.
\end{cor}

\begin{proof} Full existence when $C$ is countable is a re-statement of Proposition \ref{p-simple-existence}.
The proof of Remarks \ref{r-fe-vs-ext} (1) (see the Appendix of [Ad1]), shows that invariance, monotonicity, transitivity, normality,  symmetry,
and full existence when $C$ is countable implies extension when $B,C$ are countable.
\end{proof}

The next lemma holds in general, and has an application to the case that $T$ is stable.

\begin{lemma}  \label{l-ind[d]-pointwise}
Suppose $A,B,C\subseteq\cu K$ are countable and $A\ind[d]_C B$ in $\cu N$.  Then $A\ind[d\omega]_C \ B$.
\end{lemma}

\begin{proof}
We will use the notation from the proof of Proposition \ref{p-ind[d]-meas}.  Suppose that $A\nind[d\omega]_C \ B$.
Then the set $\sa E=\l A\ind[d]_C B\rr$ belongs to $\cu F$,  and $\mu(\sa E)<1$.  Hence there are a first-order formula $\varphi(\vec x,\vec y,Z)$,
and tuples $\vec {\bo a}\in A^{<\BN}, \vec {\bo b}\in B^{<\BN}$ such that
$$\mu(\l\varphi(\vec a,\vec b,C)\wedge\div_\varphi(\vec b,C)\rr)>0.$$
For each $k\in\BN$, let $\div^k_\varphi(\vec y,Z)$ be the part of $\div_\varphi(\vec y,Z)$ after the initial $\bigvee_k$.
Then $\div^k_\varphi(\vec y,Z)$ is a conjunctive formula,
$\l \div^k_\varphi(\vec b,C)\rr\in\cu F$, and  $\l \div_\varphi(\vec b,C)\rr=\bigcup_k\l \div^k_\varphi(\vec b,C)\rr$,
so we may find a $k\in\BN$ such that
$$r:=\mu(\l\varphi(\vec a,\vec b,C)\wedge \div^k_\varphi(\vec b,C)\rr)>0.$$
By Result \ref{f-conjunctive}, there is a countable set $\{\theta_m(\vec y,Z)\colon m\in\BN\}$  of first order formulas
closed under finite conjunction such that
$$ \cu M\models(\forall\vec y,Z)\left[\div^k_\varphi(\vec y,Z)\leftrightarrow\bigwedge_m \theta_m(\vec y,Z)\right].$$
Therefore
$$ \l\div^k_\varphi(\vec b, C)\rr=\bigcap_m \l\theta_m(\vec b,C)\rr,$$
so there exist $m(k)\in\BN$ such that
$$\mu(\l\theta_{m(k)}(\vec b,C)\wedge\neg\div^k_\varphi(\vec b,C)\rr)\le r/2.$$
Now let $\Phi(\vec x,\vec{\bo b},C)$ be the continuous formula
$$r\dotminus\mu(\l\varphi(\vec x,\vec b,C)\wedge\theta_{m(k)}(\vec b,C)\rr).$$
We have
$$\mu(\l\varphi(\vec a,\vec b,C)\wedge\theta_{m(k)}(\vec b,C)\rr)\ge r,$$
so $\Phi(\vec {\bo a},\vec {\bo b},C)=0$.

\emph{Claim.} $\Phi(\vec x,\vec {\bo b},C)$ divides over $C$ in $\cu N$.

\emph{Proof of Claim:}  Using Ramsey's theorem and the fact that $\cu M$ is saturated, one can show that for each $\omega\in \l\div^k_\varphi(\vec b,C)\rr$,
there is a  sequence $\<\vec {b'}_i\>_{i\in\BN}$ in $\cu M$ such that:
\begin{enumerate}
\item $\<\vec {b'}_i\>_{i\in\BN}$ is $C(\omega)$-indiscernible,
\item $\vec{b'}_0\equiv_{C(\omega)} \vec b(\omega)$, and
\item $\cu M\models\neg(\exists\vec x)\bigwedge_{i<k}\varphi(\vec x,\vec{b'}_i,C(\omega))$.
\end{enumerate}
By $\omega_1$-saturation for $\cu N$ and fullness, there is a sequence $\<\vec{\bo b}''_i\>_{i\in\BN}$ in $\cu K$ such that for all $\omega\in \l\div^k_\varphi(\vec b,C)\rr$,
conditions (1)--(3) above hold when $\vec b'_i=\vec b''_i(\omega)$.
By Result \ref{f-glue}, for each $i\in\BN$ there is a $\vec{\bo b}_i$ that agrees with $\vec{\bo b}'_i$ on $\l\div^k_\varphi(\vec b,C)\rr$, and agrees with $\vec{\bo b}$ elsewhere.  Then
$\<\vec{\bo b}_i\>_{i\in\BN}$ is $C$-indiscernible, and $\vec{\bo b_0}\equiv_C\vec{\bo b}$ in $\cu N$.
Consider a tuple $\vec{\bo d}\in\cu K^{|\vec x|}$, and for each $i\in\BN$ let
$$\sa D_i=\l\varphi(\vec d,\vec b_i,C)\wedge\div^k_\varphi(\vec b,C)\rr.$$
By conditions (1) and (3) above,
for all distinct elements $i_1,\ldots,i_k$ of $\BN$, we have $\mu(\bigcap_{j=1}^k \sa D_{i_j})=0$.
Therefore, by Lemma 7.5 of [EG], there exists $i\in\BN$ such that
$\mu(\sa D_i)<r/2$.  By (1) and (2) above,
$\l\div^k_\varphi(\vec b,C)\rr=\l\div^k_\varphi(\vec b_i,C)\rr$, so
$$\sa D_i=\l\varphi(\vec d,\vec b_i,C)\wedge\div^k_\varphi(\vec b_i,C)\rr.$$
Hence
$$\mu(\l\varphi(\vec d,\vec b_i,C)\wedge\theta_{m(k)}(\vec b_i,C)\rr)\le$$
$$\mu(\sa D_i)+\mu(\l(\theta_{m(k)}(\vec b_i,C)\wedge\neg\div^k_\varphi(\vec b_i,C)\rr)
<r/2+r/2=r,$$
and thus $\Phi(\vec{\bo d},\vec{\bo b}_i,C)>0.$
This proves the Claim.

It follows that $\vec{\bo a}\nind[d]_C\vec{\bo b}$ in $\cu N$, so by monotonicity, $A\nind[d]_C B$ in $\cu N$.
\end{proof}

\begin{prop}  \label{p-f-vs-omega}  Suppose $T$ is stable and let $\ind[I]=\ind[f\omega] \ \ \wedge\ind[f\BB] \ $ on $\cu N$.
\begin{enumerate}
\item $\ind[I]$ is countably based.
\item $\ind[f]\Rightarrow \ind[I]$.
\item $\ind[I] $ and $\ind[f\omega] \ \ $ are independence relations with countably local character.
\item $\ind[I] = \ind[f]$ if and only if $T$ has $\acl=\dcl$.
\end{enumerate}
\end{prop}

\begin{proof}  Since $T$ is simple, $\ind[f]=\ind[d]$ on both $\cu M$ and $\cu N$.  By Proposition \ref{p-d-cbased} and Lemma \ref{l-d-union},
over both $\cu M$ and $\cu N$, $\ind[f]$ is a countably
based independence relation with countably local character and the countable union property.

(1): $\ind[f\omega] \ \ $ and $\ind[f\BB] \ \ $ are each countably based and have the countable union property by Propositions
\ref{p-d-cbased} and \ref{p-B-stable-independence}, and Lemmas \ref{l-d-union} and \ref{l-fB-basic}.  Then $\ind[I]$ is countably based by
Proposition \ref{p-union-pair}.

(2): By (1) and Lemmas  \ref{l-cbased} (1) and \ref{l-ind[d]-pointwise}, we have $\ind[f]\Rightarrow\ind[f\omega] \ \ $.
By Lemma \ref{l-B-stable-implies}, $\ind[f]\Rightarrow\ind[f\BB] \ \ $.

(3)  The result for $\ind[f\omega] \ \ $
follows from  (2), Lemma \ref{l-cbased} (2), Proposition \ref{p-preserve-finitechar}, and Remark \ref{r-weaker}.  Proposition \ref{p-B-stable-independence} gives the corresponding
result for $\ind[f\BB] \ \ $.  It then follows easily that $\ind[I]$ satisfies the basic axioms and finite character.  By Remark \ref{r-weaker},
$\ind[I]$ also satisfies extension, local character, and countably local character.

(4)  The proof is  similar to that of Proposition \ref{p-a-vs-omega}.
Suppose $T$ has $\acl=\dcl$.  We show that $\ind[I]$ is anti-reflexive.  Let $\bo a\ind[I]_C \bo a$.  Since $\ind[f]\Rightarrow\ind[a]$ in $\cu M$, we have
$\ind[f\omega] \ \ \Rightarrow \ind[a\omega] \ \ $,
so $\bo a\ind[a\omega]_C \ \bo a$ and $\bo a\in\acl^\omega(C)=\dcl^\omega(C)$.  Also, $\ind[f\BB] \ \Rightarrow \ind[a\BB] \ $, and thus
$\bo a\ind[a\BB]_C \ \bo a$
and $\fo_\BB(\bo a C)\subseteq \acl_\BB(\bo a C)\subseteq\acl_\BB(C)$.  By Result \ref{f-acl=dcl}, $\acl_\BB(C)=\dcl_\BB(C)$.  Therefore by Result \ref{f-dcl3},
$\bo a\in\dcl(C)=\acl(C)$.  This, along with (3), shows that $\ind[I]$ is the strict independence relation on $\cu N$, so $\ind[I]=\ind[f]$.

Now suppose that $T$ does not have $\acl=\dcl$.  Take a finite set $C$ and an element $a$ in $\cu M$ such that $a\in\acl(C)\setminus\dcl(C)$.
Then in $\cu N$ we have $\ti a\ind[f\omega]_{\ti C} \ \ti a \wedge \ti a\ind[f\BB]_{\ti C} \ \ti a$ but $\ti a\notin\dcl(\ti C)=\acl(\ti C)$,
Thus $\ti a\ind[I]_{\ti C} \ti a$ but $\ti a\nind[f]_{\ti C} \ti a$, and $\ind[I]\ne\ind[f]$.
\end{proof}

\subsection{Pointwise Thorn Independence}


This subsection concerns the pointwise thorn independence relation $\ind[\th\omega] \  $, especially when $T$ is real rosy.
As usual, we note that $\thind$ over  $\cu M$ has monotonicity, so $\ind[\th\omega] \ \ $ over $\cu N$ exists.

\begin{cor} \label{c-th-omega-easyaxioms}
 The relation $\ind[\th \omega] \ \ $ over $\cu N$ is countably based, and
satisfies all the axioms for a countable independence relation except perhaps extension and local character.
Moreover, if $T$ is real rosy then $\ind[\th \omega] \ \ $ has symmetry and finite character.
\end{cor}

\begin{proof}  The basic axioms follow from Corollary \ref{c-summary}.   Symmetry follows from Proposition \ref{p-omega-cbased}.
Finite character follows from Lemma \ref{l-thorn-union} and Propositions
\ref{p-rosy-thorn-cbased} and \ref{p-pointwise-finitechar}.
\end{proof}

\begin{cor}  \label{c-th-omega-implies-a-omega}
$\ind[\th\omega] \ \ \ \Rightarrow \ind[M\omega] \ \ \ $, and $\ind[\th\omega] \ \ \ $ is pointwise anti-reflexive.
\end{cor}

\begin{proof}  This follows from Corollary \ref{c-pointwise-implies}, and the facts that
$\thind \ \Rightarrow\ind[M] \ $ and $\ind[M] \ \Rightarrow\ind[a]$.
\end{proof}

\begin{prop}  \label{p-thind-meas}
For every complete first order theory $T$, the relation $\thind$ on models of $T$ is measurable.
\end{prop}

\begin{proof}   First consider a first order formula $\psi(\vec x,\vec y,Z)$, where $\vec x, \vec y$ are tuples of variables
and $Z$ is a  countable set of variables.  For each $m$ let $\vec u_m=(u_0,\ldots,u_{m-1})$.
For each tuple $\vec b$ and countable set $C$ in the big first order model $\cu M$, $\psi(\vec x, \vec b,C)$ $\th$-divides over $C$
if and only if for some $k,m\in\BN$, $(\cu M,\vec b,C)$ satisfies the following formula $\thdiv_{\psi}(\vec y,Z)$:
$$\bigvee_k\bigvee_m (\exists \vec u_m)\bigwedge_{n\ge k}(\exists \vec y^{\, 0},\ldots,\vec y^{\, n-1})$$
$$\left[\bigwedge_{j<n}\vec y^{\, j}\equiv_{Z\vec u_m}\vec y\wedge\neg(\vec y\subseteq\acl(\vec u_m Z))\wedge
 \bigwedge_{I\subset n,|I|=k}\neg(\exists\vec x)
\bigwedge_{i\in I} \psi(\vec x,\vec y^{\, i},Z)\right].$$
We note that $\neg(y\in\acl(\vec u_m Z))$ is expressed by the conjunctive formula
$$ \bigwedge\left\{\neg\chi(y,\vec u_m,Z)\colon \chi \mbox{ algebraical}\right\},$$
so the formula  $\thdiv_{\psi}(\vec y,Z)$ is Borel-conjunctive, and is even a countable disjunction of  conjunctive formulas.

Now arrange all the first order formulas with the indicated variables
in a countable list $\<\psi_i(\vec x,\vec y,Z)\>_{i\in\BN}$.  Then a first order formula $\varphi(\vec x,\vec b,C)$ $\th$-forks over $C$
if and only if $(\cu M,\vec b,C)$ satisfies the following Borel-conjunctive formula $\thfork_\varphi(\vec y,Z)$:
$$ \bigvee_\ell\bigvee_{i_0}\cdots\bigvee_{i_\ell}
\left[\bigwedge_{j\le\ell}\thdiv_{\psi_{i_j}}(\vec y,Z)\wedge(\forall \vec x)\left[\varphi(\vec x,\vec y,Z)\rightarrow\bigvee_{j\le \ell}\psi_{i_j}(\vec x,\vec y,Z)\right]
\right].$$
By Result \ref{f-thorn-forks}, $\thind$ is definable in $\cu M$ by the Borel-conjunctive formula
$$ \neg\bigvee_{\vec x\in X^{<\BN}}\bigvee_{\vec y\in Y^{<\BN}}\bigvee_\varphi (\varphi(\vec x,\vec y,Z)\wedge \thfork_\varphi(\vec y,Z)),$$
where $X,Y,Z$ are used to index $A,B,C$.  So by Lemma \ref{l-conjunctive-measurable},  $\thind$ is measurable.
\end{proof}

The proof of Proposition \ref{p-thind-meas} gives the following

\begin{cor} For every complete first order theory $T$, the relation $\thind$ over models of $T$ is definable by the negation of a countable
disjunction of conjunctive formulas.
\end{cor}

\begin{prop}  \label{p-thorn-existence}
Suppose that $T$ is real rosy.  Then for any small $A,B$ and countable $C$, there is $A'\equiv_C A$ such that $A'\ind[\th\omega]_C \ B  $.
\end{prop}

\begin{proof}  We argue as in the proof of Proposition \ref{p-simple-existence}.
Let $A_\alpha,B_\alpha,X_\alpha$, and $X$ be as in that proof.  From the proof of Proposition \ref{p-thind-meas},
for each $\alpha$  and first order formula $\varphi(X_\alpha,B_\alpha,C)$, we have $\l \thfork_{\varphi}(B_\alpha,C)\rr\in\cu F$.
We show that the following set of statements $\Gamma(X)$ in satisfiable in $\cu N$:
\begin{itemize}
\item $\l \theta(X,C)\rr\doteq\l \theta(A,C)\rr$ for each first order formula $\theta$;
\item $ \l \thfork_{\varphi}(B_\alpha,C)\rr\sqsubseteq\l \neg \varphi(X_\alpha,B_\alpha,C)\rr$ for each $\alpha$ and $\varphi$.
\end{itemize}
By definition, any finite disjunction of formulas that $\th$-forks over $C$ again $\th$-forks over $C$.
To complete the proof we argue exactly as in the proof of Proposition \ref{p-simple-existence}, but with
$\thind$, $\ind[\th\omega] \ $, and $\thfork_\varphi$ in place of $\ind[d]$, $\ind[d\omega] \ $, and $\div_\varphi$.
\end{proof}

\begin{cor}  \label{c-thorn-existence}
If $T$ is real rosy, then $\ind[\th\omega] \ \ $ satisfies full existence when $C$ is countable, and satisfies extension
when $B, C$ are countable.
\end{cor}

\begin{proof} Like the proof of Corollary \ref{c-simple-existence}.
\end{proof}

\section*{References}


\vspace{2mm}

[Ad1]  Hans Adler. Explanation of Independence.  PH. D. Thesis, Freiburg, AxXiv:0511616 (2005).

[Ad2]  Hans Adler.  A Geometric Introduction to Forking and Thorn-forking.  J. Math. Logic 9 (2009), 1-21.

[Ad3]  Hans Adler.  Thorn Forking as Local Forking.  Journal of Mathematical Logic 9 (2009), 21-38.

[AGK]  Uri Andrews, Isaac Goldbring, and H. Jerome Keisler.  Definability in Randomizations.
To appear, Annals of Pure and Applied Logic.  Available online at www.math.wisc.edu/$\sim$Keisler.

[AK]  Uri Andrews and H. Jerome Keisler.  Separable Randomizations of Models.  To appear, Journal of Symbolic Logic.  Available online at
www.math.wisc.edu/$\sim$Keisler.

[Be1]  Ita\"i{} Ben Yaacov. Schrodinger's Cat.  Israel J. Math. 153 (2006),  157-191.

[Be2]  Ita\"i{} Ben Yaacov.  On Theories of Random Variables.  Israel J. Math 194 (2013), 957-1012.

[Be3]  Ita\"i{} Ben Yaacov.  Simplicity in Compact Abstract Theories.  Journal of Mathematical Logic 3 (2003), 163-191.

[BBHU]  Ita\"i{} Ben Yaacov, Alexander Berenstein,
C. Ward Henson and Alexander Usvyatsov. Model Theory for Metric Structures.  In Model Theory with Applications to Algebra and Analysis, vol. 2,
London Math. Society Lecture Note Series, vol. 350 (2008), 315-427.

[BK] Ita\"i{} Ben Yaacov and H. Jerome Keisler. Randomizations of Models as Metric Structures.
Confluentes Mathematici 1 (2009), pp. 197-223.

[BU] Ita\"i{} Ben Yaacov and Alexander Usvyatsov. Continuous first order logic and local stability. Transactions of the American
Mathematical Society 362 (2010), no. 10, 5213-5259.

[CK]  C.C.Chang and H. Jerome Keisler.  Model Theory.  Dover 2012.

[EG]  Clifton Ealy and Isaac Goldbring.  Thorn-Forking in Continuous Logic.  Journal of Symbolic Logic
77 (2012), 63-93.

%
%

[GL]  Isaac Goldbring and Vinicius Lopes.  Pseudofinite and Pseudocompact Metric Structures.  To appear, Notre Dame Journal of Formal Logic.
Available online at www.homepages.math.uic.edu/$\sim$isaac.

[GIL]  Rami Grossberg, Jose Iovino, and Oliver Lessmann, A Primer of Simple Theories, Archive for Ma6th. Logic 41 (2002), 541-580.

[Ke1] H. Jerome Keisler.  Randomizing a Model.  Advances in Math 143 (1999),
124-158.

[Ke2] Finite Approximations of Infinitely Long Formulas, Theory of Models,
Amsterdam 1965, pp. 158-169.

[On]  Alf Onshuus, Properties and Consequences of Thorn Independence. J. Symbolic Logic 71 (2006), 1-21.

\end{document}